\documentclass[reqno,a4paper,11pt]{amsart}

\usepackage[utf8]{inputenc}
\usepackage{amsmath}
\usepackage{amssymb}
\usepackage{amsthm}
\usepackage{mathtools}

\usepackage{times}
\usepackage{amsaddr}
\usepackage{enumitem}
\usepackage{cite}
\usepackage{tikz}
\usepackage{cleveref}
\usepackage[english]{babel}
\usepackage[T1]{fontenc}
\usepackage{fourier}
\usepackage{fullpage}
\usepackage{caption}

\newcommand{\whp}{whp}
\newcommand{\Whp}{Whp}
\newcommand{\prob}[1]{\mathbb{P}\left[#1\right]} 
\newcommand{\probLarge}[1]{\mathbb{P}\big[#1\big]} 
\newcommand{\condprob}[2]{\mathbb{P}\left[#1 \;\middle|\; #2\right]}
\newcommand{\variance}[1]{\mathbb{V}\left[#1\right]}
\newcommand{\expec}[1]{\mathbb{E}\left[#1\right]} 

\def\Erdos{Erd\H{o}s}
\def\Renyi{R\'enyi}
\def\Luczak{\L{}uczak}
\def\ER{\Erdos-\Renyi}
\def\Bollobas{Bollob\'{a}s}

\newtheorem{thm}{Theorem}[section]
\newtheorem{prop}[thm]{Proposition}
\newtheorem{coro}[thm]{Corollary}
\newtheorem{lem}[thm]{Lemma}

\newtheorem{question}[thm]{Question}
\theoremstyle{remark}
\newtheorem{remark}[thm]{Remark}
\theoremstyle{definition}
\newtheorem{definition}[thm]{Definition}
\newcommand{\proofof}[1]{\subsection{Proof of \Cref{#1}}}

\crefname{thm}{theorem}{theorems}
\crefname{prop}{proposition}{propositions}
\crefname{coro}{corollary}{corollaries}
\crefname{lem}{lemma}{lemmas}
\crefname{definition}{definition}{definitions}
\crefname{question}{question}{questions}

\newtheoremstyle{claim}
{}
{}
{\itshape}
{}
{\bf}
{.}
{.5em}
{}
\theoremstyle{claim}

\crefname{claim}{claim}{claims}

\def\N{\mathbb{N}} 
\def\R{\mathbb{R}} 
\def\ur{\in_R} 
\newcommand{\rounddown}[1]{\left\lfloor#1\right\rfloor} 

\def\twoconcentration{D}

\def\bin{\mathcal{B}}
\def\ball{B}
\def\nbins{n} 
\def\nballs{k} 
\def\location{\mathbf{A}} 
\def\locationBit{A} 
\def\loadvector{\mathbf{\lambda}} 
\def\load{\lambda} 
\def\maxload{\lambda^\ast} 
\DeclareMathOperator{\BB}{BB}
\DeclareMathOperator{\MBB}{M}
\newcommand{\binsandballs}[2]{\BB\left(#1, #2\right)}
\newcommand{\maxbinsandballs}[2]{\MBB\left(#1, #2\right)}

\newcommand{\concentration}[2]{\nu\left(#1, #2\right)} 
\newcommand{\specialconcentration}[1]{\nu\left(#1\right)} 
\def\Concentration{\nu}

\def\ntrees{t} 
\def\rootrv{Z} 
 
\def\rootvectorbit{Y}

\newcommand{\smallo}[1]{o\left(#1\right)}
\newcommand{\bigo}[1]{O\left(#1\right)}
\newcommand{\smallomega}[1]{\omega\left(#1\right)}
\newcommand{\Th}[1]{\Theta\left(#1\right)}

\def\planargraph{P} 
\def\planarclass{\mathcal{P}} 

\def\multigraph{M} 

\def\forest{F} 
\def\forestclass{\mathcal{F}} 

\def\nocomplex{U} 
\def\nocomplexclass{\mathcal{U}} 

\newcommand{\maxdegree}[1]{\Delta\left(#1\right)} 
\newcommand{\maxdegreeLarge}[1]{\Delta\big(#1\big)} 
\newcommand{\largestcomponent}[1]{L\left(#1\right)} 
\def\Largestcomponent{L} 
\newcommand{\rest}[1]{R\left(#1\right)} 
\def\Rest{R} 
\def\degreesequence{\mathbf{d}}
\def\ratio{\mu}
\def\niv{k}
\def\nie{l}
\def\md{d}
\newcommand{\degree}[2]{d_{#2}\left(#1\right)} 
\newcommand{\vertexSet}[1]{V\left(#1\right)} 
\newcommand{\edgeSet}[1]{E\left(#1\right)} 
\newcommand{\numberVertices}[1]{v\left({#1}\right)} 
\newcommand{\numberEdges}[1]{e\left({#1}\right)} 
\newcommand{\numberVerticesLarge}[1]{v\big({#1}\big)} 
\newcommand{\numberEdgesLarge}[1]{e\big({#1}\big)} 

\newcommand{\core}[1]{C\left(#1\right)} 
\newcommand{\complexpart}[1]{Q\left(#1\right)} 
\newcommand{\restcomplex}[1]{U\left(#1\right)} 
\newcommand{\restcomplexLarge}[1]{U\big(#1\big)}
\def\Restcomplex{U} 
\def\Complexlargestcore{Q_L} 
\newcommand{\complexlargestcore}[1]{Q_L\left(#1\right)}
\newcommand{\complexlargestcoreLarge}[1]{Q_L\big(#1\big)}
\def\Complexrestcore{Q_S} 
\newcommand{\complexrestcore}[1]{Q_S\left(#1\right)}
\newcommand{\complexrestcoreLarge}[1]{Q_S\big(#1\big)}
\def\complexclass{\mathcal{Q}}
\def\complexgraph{Q}
\def\funcL{\ell}
\def\funcR{r}

\def\cl{\mathcal{A}} 
\def\func{\Phi} 
\def\seq{\mathbf{a}} 
\def\term{a} 
\newcommand{\condGraph}[2]{#1 \mid #2} 
\def\randomGraph{A}
\def\property{\mathcal{R}}

\newcommand{\contiguous}[2]{#1 \triangleleft #2} 

\def\prueferseqence{\psi} 
\def\prueferinvers{\psi^{-1}} 
\newcommand{\sequences}[2]{\mathcal{S}\left(#1, #2\right)} 
\newcommand{\frequency}[2]{\#\left(#1, #2\right)} 

\newcommand{\setbuilder}[2]{\left\{#1 \mid #2\right\}} 

\newcommand{\lessorequal}{\hspace{0.06cm}\leq\hspace{0.06cm}}
\newcommand{\greaterorequal}{\hspace{0.06cm}\geq\hspace{0.06cm}}
\newcommand{\equal}{\hspace{0.06cm}=\hspace{0.06cm}}
\newcommand{\greater}{\hspace{0.06cm}>\hspace{0.06cm}}
\newcommand{\defined}{\hspace{0.06cm}:=\hspace{0.06cm}}

\newcommand{\transformation}[2]{#1\rightarrow #2}

\title{Concentration of maximum degree in random planar graphs}
\author{Mihyun Kang, Michael Missethan}
\address{Institute of Discrete Mathematics, Graz University of Technology, Steyrergasse 30, 8010 Graz, Austria}
\email{\{kang,missethan\}@math.tugraz.at}
\thanks{Supported by Austrian Science Fund (FWF): I3747 and W1230}

\begin{document}

\begin{abstract}
Let $\planargraph(n,m)$ be a graph chosen uniformly at random from the class of all planar graphs on vertex set $[n]:=\left\{1, \ldots, n\right\}$ with $m=m(n)$ edges. We show that in the sparse regime, when $m/n\leq 1$, with high probability the maximum degree of $\planargraph(n,m)$ takes at most two different values. In contrast, this is not true anymore in the dense regime, when $m/n>1$, where the maximum degree of $\planargraph(n,m)$ is not concentrated on any subset of $[n]$ with bounded size.
\end{abstract}

\maketitle

\section{Introduction and results}\label{sec:intro}

\subsection{Motivation}\label{subsec:motivation}
The \ER\ random graph $G(n,m$), introduced by \Erdos\ and \Renyi\ \cite{erdoes1,erdoes2}, is a graph chosen uniformly at random from the class $\mathcal{G}(n,m)$ of all vertex-labelled graphs on vertex set $[n]:=\left\{1, \ldots, n\right\}$ with $m=m(n)$ edges, denoted by $G(n,m)\ur \mathcal{G}(n,m)$. Since its introduction $G(n,m)$, together with the closely related binomial random graph $G(n,p)$, has been intensively studied (see e.g. \cite{rg1, rg2, rg3}). A remarkable feature of this model is the \lq concentration\rq\ of many graph parameters. That is, with high probability (meaning with probability tending to 1 as $n$ tends to infinity, {\em \whp\ } for short) certain graph parameters in $G(n,m)$ lie in \lq small\rq\ intervals, which only depend on $n$ and $m$.

The graph parameter we will focus on in this paper is the maximum degree of a graph $H$, denoted by $\maxdegree{H}$. \Erdos\ and \Renyi\ \cite{erdoes1} were the first to consider $\maxdegree{G(n,m)}$ and since then, many results on $\maxdegree{G(n,m)}$ and, more generally, the degree sequence of $G(n,m)$ were obtained (see e.g. \cite{max_degree1,max_degree2,max_degree3,degree_sequence1,degree_sequence2,degree_sequence3,vertices_given_degree}). A particularly interesting result by \Bollobas\ \cite{vertices_given_degree} is that $m\sim n\log n$ is a \lq threshold\rq\ for the concentration of $\maxdegree{G(n,m)}$. More formally, \whp\ $\maxdegree{G(n,m)}$ takes one of two values when $m=\smallo{n \log n}$, while $\maxdegree{G(n,m)}$ is not concentrated on any subset of $[n]$ with bounded size when $m=\smallomega{n \log n}$ and $m=\smallo{n^2}$.

\begin{thm}[\hspace{1sp}{\cite{vertices_given_degree}}]\label{thm:max_degree_ergraph}
Let $m=m(n)=\smallo{n\log n}$ and $G=G(n,m)\ur \mathcal G(n,m)$. Then there exists a $\twoconcentration=\twoconcentration(n)\in \N$ such that \whp\ $\maxdegree{G}\in\left\{\twoconcentration, \twoconcentration+1\right\}$.
\end{thm}

\begin{thm}[\hspace{1sp}{\cite{vertices_given_degree}}]\label{thm:max_degree_ergraph2}
	Let $m=m(n)=\smallomega{n\log n}$, $m=\smallo{n^2}$, and $G=G(n,m)\ur \mathcal G(n,m)$. If $I=I(n)\subseteq[n]$ is such that \whp\ $\maxdegree{G}\in I$, then $\left|I\right|=\smallomega{1}$.
\end{thm}

We note that \Bollobas\ \cite{vertices_given_degree} actually considered the binomial random graph $G(n,p)$. But by using standard tools relating $G(n,m)$ and $G(n,p)$ (see e.g. \cite[Section 1.1]{rg1}) one can translate his results as stated in \Cref{thm:max_degree_ergraph,thm:max_degree_ergraph2}.

In recent decades various models of random graphs have been introduced by imposing additional constraints to $G(n,m)$, e.g. degree restrictions or topological constraints. In particular, random planar graphs and related structures, like random graphs on surfaces and random planar maps, have attained considerable attention \cite{chap,dks,surface,mcd,chap2,gim,planar,msw,maxdegree_planar1,max_degree_vertex_model,planar_map1,planar_map2,planar_map3,planar1,planar2,planar3,planar4,planar5,planar6,planar7,planar8,planar9,planar10}. McDiarmid and Reed \cite{maxdegree_planar1} considered the so-called $n$-vertex model for random planar graphs, that is, a graph $\planargraph(n)$ chosen uniformly at random from the class of all vertex-labelled planar graphs on vertex set $[n]$. They proved that \whp\ $\maxdegree{\planargraph(n)}=\Th{\log n}$. Later Drmota, Giménez, Noy, Panagiotou, and Steger \cite{max_degree_vertex_model} used tools from analytic combinatorics and Boltzmann sampling techniques to show that \whp\ $\maxdegree{\planargraph(n)}$ is concentrated in an interval of length $\bigo{\log\log n}$. 

A more natural generalisation of $G(n,m)$ seems to be the random planar graph $\planargraph(n,m)$, which is a graph chosen uniformly at random from the class $\planarclass(n,m)$ of all vertex-labelled planar graphs on vertex set $[n]$ with $m=m(n)$ edges, denoted by $\planargraph(n,m)\ur \planarclass(n,m)$. The random planar graph $\planargraph(n,m)$ has been studied separately for the \lq sparse\rq\ regime where $m\leq n+\smallo{n}$ (see \cite{planar,surface}) and the \lq dense\rq\ regime where $m=\rounddown{\ratio n}$ for a constant $\ratio\in (1,3)$ (see e.g. \cite{gim}). In this paper we show, in the flavour of \Cref{thm:max_degree_ergraph,thm:max_degree_ergraph2}, that in the sparse regime \whp\ $\maxdegree{\planargraph(n,m)}$ takes one of two values (see \Cref{thm:main,thm:main_sub}), while in the dense regime $\maxdegree{\planargraph(n,m)}$ is not concentrated on any subset of $[n]$ with bounded size (see \Cref{thm:main_dense}).

\subsection{Main results}\label{subsec:main}
In order to state our main results, we need the following definition, where we denote by $\log$ the natural logarithm.
\begin{definition}\label{def:nu}
	Let $\Concentration:\N^2\to \R^+$ be a function such that $\concentration{\nbins}{\nballs}$ is the unique positive zero of
	\begin{align*}
	f(x)=f_{\nbins, \nballs}(x)\defined x\log k+x-(x+1/2)\log x-(x-1)\log n.
	\end{align*}
	In case of $\nbins=\nballs$, we write $\specialconcentration{\nbins}:=\concentration{\nbins}{\nbins}$.
\end{definition}
In \Cref{sub:well_definedness} we will prove that the function $\Concentration$ is well defined, i.e. $f$ has a unique positive zero. In \Cref{lem:nu} we will provide some important properties of $\Concentration$ and in \Cref{sec:balls_bins} we motivate the definition of $\Concentration$ in the context of the balls-into-bins model.

We distinguish different cases according to which \lq region\rq\ the edge density falls into. The first regime which we consider is when $m\leq n/2+\bigo{n^{2/3}}$.

\begin{thm}\label{thm:main_sub}
Let $\planargraph=\planargraph(n,m)\ur \planarclass(n,m)$, $m=m(n)\leq n/2+\bigo{n^{2/3}}$, and $\varepsilon>0$. Then we have \whp\ $\rounddown{\concentration{n}{2m}-\varepsilon}\lessorequal \maxdegree{\planargraph}\lessorequal \rounddown{\concentration{n}{2m}+\varepsilon}$. In particular, \whp\ $\maxdegree{P}\in \left\{\twoconcentration, \twoconcentration+1\right\}$, where $\twoconcentration=\twoconcentration(n):=\rounddown{\concentration{n}{2m}-1/3}$.
\end{thm}

Next, we consider the case where $m\geq n/2+\smallomega{n^{2/3}}$ is such that $m\leq n+n^{1-\delta}$ for some $\delta>0$. Kang and \Luczak\ \cite{planar} and Kang, Moßhammer, and Sprüssel \cite{surface} showed that, in contrast to the case when $m\leq n/2+\bigo{n^{2/3}}$, in this regime \whp\ the largest component of $\planargraph=\planargraph(n,m)$ contains significantly more vertices than the second largest component. Therefore, we provide a concentration result on the maximum degree not only for $\planargraph$, but also for the largest component $\largestcomponent{\planargraph}$ of $\planargraph$ and the \lq rest\rq\ $\rest{\planargraph}:=\planargraph\setminus\largestcomponent{\planargraph}$. We will see that $\maxdegree{\largestcomponent{\planargraph}}$ and $\maxdegree{\rest{\planargraph}}$ are strongly concentrated around $\specialconcentration{\funcL}$ and $\specialconcentration{\funcR}$ for suitable $\funcL=\funcL(n)$ and $\funcR=\funcR(n)$, respectively. 

\begin{definition}\label{def:functionLR}
Let $m=m(n)\geq n/2+\smallomega{n^{2/3}}$ be such that $m\leq n+n^{1-\delta}$ for some $\delta>0$. Then we define $\funcL=\funcL(n)$ and $\funcR=\funcR(n)$ as follows.	
\begin{center}
	\def\arraystretch{1.25}
	\begin{tabular}{l|c c }
	 \multicolumn{1}{c|}{ranges of $m$}	& $~~\funcL~~$ & $\funcR$
		\\
		\hline
		(a)~ $m=n/2+s$ for $s=s(n)>0$ such that $s=\smallo{n}$ and $s^3n^{-2}\to \infty$ & $s$ & $n$ \\
		(b)~ $m=d n/2$ where $d=d(n)$ tends to a constant in $\left(1,2\right)$ & $n$ & $n$ \\
		(c)~ $m=n+t$ for $t=t(n)<0$ such that $t=\smallo{n}$ and $t^5n^{-3}\to -\infty$ & $n$ & $\left|t\right|$ \\
		(d)~ $m=n+t$ for $t=t(n)$ such that $t^5n^{-3}$ tends to a constant in $\R$ & $n$ & $n^{3/5}$ \\
		(e)~ $m=n+t$ for $t=t(n)>0$ such that $t=\smallo{n}$ and $t^5n^{-3}\to \infty$ & $n$ & $n^{3/2}t^{-3/2}$
	\end{tabular}
\end{center}
\end{definition}

We note that $\funcL$ and $\funcR$ are chosen such that \whp\ the number of vertices in $\largestcomponent{\planargraph}$ and $\rest{\planargraph}$ are $\Th{\funcL}$ and $\Th{\funcR}$, respectively, according to results in \cite{surface,planar}. Throughout the paper, we will assume that if $m=m(n)\geq n/2+\smallomega{n^{2/3}}$ is such that $m\leq n+n^{1-\delta}$ for some $\delta>0$, then $m=m(n)$ lies in one of the five regimes considered in \Cref{def:functionLR}, which is due to the critical phenomena observed in random planar graphs. Our next result states that in all these cases $\maxdegree{\planargraph}$, $\maxdegree{\largestcomponent{\planargraph}}$, and $\maxdegree{\rest{\planargraph}}$ are strongly concentrated. 

\begin{thm}\label{thm:main}
Let $\planargraph=\planargraph(n,m)\ur \planarclass(n,m)$, $\Largestcomponent=\largestcomponent{\planargraph}$ be the largest component of $\planargraph$, and $\Rest=\planargraph\setminus\Largestcomponent$. Assume $m=m(n)\geq n/2+\smallomega{n^{2/3}}$ is such that $m\leq n+n^{1-\delta}$ for some $\delta>0$. Let $\funcL=\funcL(n)$ and $\funcR=\funcR(n)$ be as in \Cref{def:functionLR} and $\varepsilon>0$. Then \whp
\begin{enumerate}
	\item\label{thm:main1} $\rounddown{\specialconcentration{\funcL}-\varepsilon}+1\lessorequal \maxdegree{\Largestcomponent}\lessorequal \rounddown{\specialconcentration{\funcL}+\varepsilon}+1$;
	\item\label{thm:main2} $\rounddown{\specialconcentration{\funcR}-\varepsilon}\lessorequal \maxdegree{\Rest}\lessorequal \rounddown{\specialconcentration{\funcR}+\varepsilon}$.
\end{enumerate}
In particular, \whp\ $\maxdegree{P}\in \left\{\twoconcentration, \twoconcentration+1\right\}$, where $\twoconcentration=\twoconcentration(n):=\max\left\{\rounddown{\specialconcentration{\funcL}+2/3}, \rounddown{\specialconcentration{\funcR}-1/3}\right\}$.
\end{thm}

For example, \Cref{thm:main} says that if $m=n/2+s$ for $s=s(n)>0$ such that $s=\smallo{n}$ and $s^3n^{-2}\to \infty$, known as the weakly supercritical regime, then \whp\ $\rounddown{\specialconcentration{s}-\varepsilon}+1\lessorequal \maxdegree{\Largestcomponent}\lessorequal \rounddown{\specialconcentration{s}+\varepsilon}+1$. In contrast, if $m=d n/2$ where $d=d(n)$ tends to a constant in $\left(1,2\right)$, which is the so-called intermediate regime, then \whp\ $\rounddown{\specialconcentration{n}-\varepsilon}+1\lessorequal \maxdegree{\Largestcomponent}\lessorequal \rounddown{\specialconcentration{n}+\varepsilon}+1$.

Combining \Cref{thm:main,thm:main_sub} we can determine the asymptotic order of $\maxdegree{\planargraph}$ in the sparse regime.
\begin{coro}\label{cor:maxdegree}
	Let $\planargraph=\planargraph(n,m)\ur \planarclass(n,m)$ and assume $m=m(n)$ is such that $\liminf_{n \to \infty} m/n>0$ and $m\leq n+n^{1-\delta}$ for some $\delta>0$. Then \whp
	\begin{align*}
	\maxdegree{\planargraph}=\left(1+\smallo{1}\right)\frac{\log n}{\log \log n}.
	\end{align*}
\end{coro}

It is well known that when $m\leq n/2+\bigo{n^{2/3}}$, the probability that $G(n,m)$ is planar is bounded away from 0 (see e.g. \Cref{thm:non_complex} and \cite{birth_giant,luczak_pittel_wierman,prob_planarity}) and therefore, $\planargraph(n,m)$ \lq behaves\rq\ asymptotically like $G(n,m)$. However, this is not the case anymore when $m\geq n/2+\smallomega{n^{2/3}}$, since then \whp\ $G(n,m)$ is not planar (see \cite{luczak_pittel_wierman,prob_planarity}). \Cref{thm:main} reveals the following, perhaps surprising, difference between $\planargraph=\planargraph(n,m)$ and $G=G(n,m)$ in the case that $m=m(n)=d n/2$ where $d=d(n)$ tends to a constant in $\left(1,2\right)$. Roughly speaking, the maximum degrees $\maxdegree{\planargraph}$, $\maxdegree{\largestcomponent{\planargraph}}$, and $\maxdegree{\rest{\planargraph}}$ are {\em independent} of $d=d(n)$. Furthermore, $\maxdegree{\largestcomponent{\planargraph}}$ and $\maxdegree{\rest{\planargraph}}$ typically differ at most by two.

\begin{coro}\label{cor:independent}
Let $\planargraph=\planargraph(n,m)\ur \planarclass(n,m)$, $\Largestcomponent=\largestcomponent{\planargraph}$ be the largest component of $\planargraph$, and $\Rest=\planargraph\setminus\Largestcomponent$. There exists a $\twoconcentration=\twoconcentration(n)$ such that for all $m=m(n)=d n/2$ where $d=d(n)$ tends to a constant in $\left(1,2\right)$, we have \whp\ 
\begin{align*}
\maxdegree{\Largestcomponent}&\in \left\{\twoconcentration, \twoconcentration+1\right\};\\
\maxdegree{\Rest}&\in \left\{\twoconcentration-1, \twoconcentration\right\};\\
\maxdegree{\planargraph}&\in \left\{\twoconcentration, \twoconcentration+1\right\}.
\end{align*}
In particular, \whp\ $\maxdegree{\Rest}\leq\maxdegree{\Largestcomponent}\leq \maxdegree{\Rest}+2$.
\end{coro}

In contrast to \Cref{cor:independent}, $G=G(n,m)$ exhibits a perhaps more intuitive behaviour. If the average degree $d=2m/n$ is growing, then $\maxdegree{G}$ and $\maxdegree{\largestcomponent{G}}$ are increasing, while $\maxdegree{\rest{G}}$ is decreasing. As a consequence, $\maxdegree{\largestcomponent{G}}$ is typically much larger than $\maxdegree{\rest{G}}$. 

\begin{prop}\label{pro:ER}
For $i=1,2$, let $m_i=m_i(n)=d_i n/2$ where $d_i=d_i(n)$ tends to a constant $c_i>1$ and $G_i=G(n,m_i)\ur \mathcal G(n,m_i)$. If $c_1<c_2$ and $G_1$ and $G_2$ are chosen independently from each other, then \whp\ $\maxdegree{G_2}-\maxdegree{G_1}$, $\maxdegree{\largestcomponent{G_2}}-\maxdegree{\largestcomponent{G_1}}$, 	$\maxdegree{\rest{G_1}}-\maxdegree{\rest{G_2}}$, and $\maxdegree{\largestcomponent{G_1}}-\maxdegree{\rest{G_1}}$ are strictly positive and of order $\Th{\log n/\left(\log \log n\right)^2}$.
\end{prop}

\Cref{pro:ER} follows by a generalised version of \Cref{thm:max_degree_ergraph} (see \Cref{thm:G_n_m_bins_balls}\ref{thm:G_n_m_bins_balls2}) and classical results on the \ER\ random graph $G(n,m)$. (For the sake of completeness, we provide a sketch of the proof of \Cref{pro:ER} in \Cref{sec:proof_er}.)

Finally, we consider the dense case when $m=m(n)=\rounddown{\ratio n}$ for $\ratio \in(1,3)$ and show that in this regime $\maxdegree{\planargraph}$ is not concentrated on any subset of $[n]$ with bounded size. 
\begin{thm}\label{thm:main_dense}
	Let $\planargraph=\planargraph(n,m)\ur \planarclass(n,m)$ and assume $m=m(n)=\rounddown{\ratio n}$ for $\ratio \in(1,3)$. If $I=I(n)\subseteq[n]$ is such that \whp\ $\maxdegree{\planargraph}\in I$, then $\left|I\right|=\smallomega{1}$.
\end{thm}

We note that in a planar graph on $n\geq 3$ vertices there are at most $3n-6$ edges, while a general (not necessarily planar) graph can have up to $\binom{n}{2}$ edges. In view of this fact, it seems natural that the \lq transition\rq\ from the two-point concentration to the non-concentration of the maximum degree occurs {\em much earlier} in $\planargraph(n,m)$ than in $G(n,m)$, namely at $m\sim n$ in $\planargraph(n,m)$ (cf. \Cref{thm:main_sub,thm:main,thm:main_dense}) instead of $m\sim n\log n$ in $G(n,m)$ (cf. \Cref{thm:max_degree_ergraph,thm:max_degree_ergraph2}). It is worth noting that the \lq threshold\rq\ where the number of vertices outside the largest component drops from linear to sublinear is $m\sim n$ for the random planar graph $\planargraph(n,m)$, while it is $m \sim n\log n$ in the case of $G(n,m)$.

\subsection{Outline of the paper} 
The rest of the paper is structured as follows. After giving the necessary definitions, notations, and concepts in \Cref{sec:prelim}, we provide our proof strategy in \Cref{sec:strategy}. \Cref{sec:balls_bins} is devoted to the balls-into-bins model, which we use in \Cref{sec:ergraph,sec:forests} to show concentration of the maximum degree in the \ER\ random graph, in a random graph without complex components, and in a random forest with specified roots, respectively. In \Cref{sec:proof,sec:planar_dense} we provide the proofs of our main results. Subsequently in \Cref{sec:nu}, we consider the function $\concentration{\nbins}{\nballs}$ introduced in \Cref{def:nu} in more detail. Finally in \Cref{sec:discussion}, we discuss a possible generalisation of our results and various questions that remain unanswered.

\section{Preliminaries}\label{sec:prelim}
\subsection{Notations for graphs}
We consider only undirected graphs or multigraphs and we always assume that the graphs are vertex-labelled.
\begin{definition}
	Given a (simple or multi) graph $H$ we denote by
	\begin{itemize}
		\item 
		$\vertexSet{H}$ the vertex set of $H$ and
		\item[]
		$\numberVertices{H}$ the order of $H$, i.e. the number of vertices in $H$;
		\item 
		$\edgeSet{H}$ the edge set of $H$ and
		\item[]
		$\numberEdges{H}$ the size of $H$, i.e. the number of edges in $H$;	
		\item 
		$\largestcomponent{H}$ the largest component of $H$ (if there are more than one largest component, we pick that which contains the vertex with smallest label);
		\item
		$\rest{H}:=H \setminus \largestcomponent{H}$ the graph obtained from $H$ by deleting the largest component;
		\item
		$\degree{v}{H}$ the degree of a vertex $v\in\vertexSet{H}$. If $\vertexSet{H}=[n]$, then we call $\left(\degree{1}{H}, \ldots, \degree{n}{H}\right)$ the degree sequence of $H$.
	\end{itemize}
\end{definition}
\begin{definition}\label{def:graph_class}
	Given a class $\cl$ of graphs (e.g. the class of planar graphs), we denote by 
	$\cl(n)$ the subclass of $\cl$ containing the graphs on vertex set $[n]$ 
	and by $\cl(n,m)$ the subclass of $\cl$ containing the graphs on vertex set $[n]$ with $m$ edges, respectively. We write $\randomGraph(n)\ur \cl(n)$ for a graph chosen uniformly at random from $\cl(n)$
	and $\randomGraph(n,m)\ur \cl(n,m)$ for a graph chosen uniformly at random from $\cl(n,m)$, respectively.
\end{definition}

\subsection{Random variables and asymptotic notation}
\begin{definition}
	Let $S$ be a finite set and let $Y$ and $Z$ be random variables with values in $S$. Then we say that $Y$ is {\em distributed like} $Z$, denoted by $Y\sim Z$, if for all $x \in S$ we have $\prob{Y=x}=\prob{Z=x}$.
\end{definition}
Throughout this paper, we use the standard Landau notation and all asymptotics are taken with respect to $n$, i.e. when $n\to \infty$. In order to express that two random variables have asymptotically a \lq similar\rq\ distribution, we use the notion of contiguity.
\begin{definition}
	For each $n\in\N$, let $S=S(n)$ be a finite set and let $Y=Y(n)$ and $Z=Z(n)$ be random variables with values in $S$. We say that $Z$ is {\em contiguous} with respect to $Y$, denoted by $\contiguous{Z}{Y}$, if for all sequences $I=I(n)\subseteq S(n)$
	\begin{align*}
		\left(\lim\limits_{n\to \infty}\prob{Y\in I}=1\right) ~\implies~ \left(\lim\limits_{n\to \infty}\prob{Z\in I}=1\right).
	\end{align*}
\end{definition}

\subsection{Complex part and core}\label{sub:decomposition}
We say that a component of a graph $H$ is {\em complex} if it has at least two cycles. The union of all complex components is called the {\em complex part} $\complexpart{H}$. We call the graph $H$ {\em complex} if all its components are complex. The union of all non-complex components is the {\em non-complex part} $\restcomplex{H}:=H\setminus \complexpart{H}$. The {\em core} $\core{H}$, also known as the 2-core, is the maximal subgraph of $\complexpart{H}$ of minimum degree at least two. We observe that the core $\core{H}$ can be obtained from $H$ by first removing vertices of degree one recursively and then deleting isolated cycles. We emphasise that the core $\core{H}$ is not necessarily connected. We denote by $\complexlargestcore{H}$ the component of $\complexpart{H}$ containing the largest component of the core $\largestcomponent{\core{H}}$. The rest of the complex part is denoted by $\complexrestcore{H}:=\complexpart{H}\setminus\complexlargestcore{H}$. We call $\complexlargestcore{H}$ and $\complexrestcore{H}$ the {\em large complex part} and the {\em small complex part}, respectively. We note that the number of vertices in $\complexlargestcore{H}$ is not necessarily larger than in $\complexrestcore{H}$, but it will be true in most cases we consider. Using this decomposition we can split $H$ into the three disjoint parts $\complexlargestcore{H}$, $\complexrestcore{H}$, and $\restcomplex{H}$. Moreover, we have the relations $\core{\complexlargestcore{H}}=\largestcomponent{\core{H}}$ and $\core{\complexrestcore{H}}=\rest{\core{H}}$.

Later we will construct the large complex part, the small complex part, and the non-complex part of a random planar graph independently of each other. To that end, we will use the following two graph classes.
\begin{definition}\label{def:random_complex_part}
	Let $C$ be a core, i.e. a graph with minimum degree at least two containing no isolated cycles, and $q\in \N$. Then we denote by $\complexclass(C,q)$ the class consisting of {\em complex} graphs having core $C$ and vertex set $[q]$. We let $\complexgraph(C,q)\ur\complexclass(C,q)$ be a graph chosen uniformly at random from this class.
\end{definition}

\begin{definition}\label{def:nocomplex}
We denote by $\nocomplexclass$ the class consisting of all graphs without complex components. For $n,m\in\N$ we let $\nocomplexclass(n,m)\subseteq\nocomplexclass$ be the subclass of all graphs on vertex set $[n]$ with $m$ edges and we write $\nocomplex(n,m)\ur \nocomplexclass(n,m)$ for a graph chosen uniformly at random from $\nocomplexclass(n,m)$.
\end{definition}

\begin{remark}\label{rem:connection_conditional}
Let $C$ be a core, $q, n, m \in\N$, and $H\in\complexclass(C,q)$ be a fixed graph. Then there are precisely $\left|\nocomplexclass(u,w)\right|$ many graphs $H'\in\planarclass(n,m)$ whose complex part is $H$, where $u:=n-q$ and $w:=m-\numberEdges{C}+\numberVertices{C}-q$. As this number is independent of $H\in\complexclass(C,q)$, there is a nice relation between the complex part $\complexpart{\planargraph}$ of the random planar graph $\planargraph=\planargraph(n,m)$ and $\complexgraph(C,q)\ur\complexclass(C,q)$, the latter being as in \Cref{def:random_complex_part}: Conditioned on the event that the core $\core{\planargraph}$ is $C$ and $\numberVertices{\complexpart{\planargraph}}=q$, the complex part $\complexpart{\planargraph}$ is distributed like $\complexgraph(C,q)$. Similarly, for fixed $\tilde{n},\tilde{m},n,m \in \N$ let $\restcomplex{\planargraph}$ be the non-complex part of $\planargraph$ and $\nocomplex(\tilde{n},\tilde{m})\ur \nocomplexclass(\tilde{n},\tilde{m})$ be as in \Cref{def:nocomplex}. Then, conditioned on the event that $\numberVertices{\restcomplex{\planargraph}}=\tilde{n}$ and $\numberEdges{\restcomplex{\planargraph}}=\tilde{m}$, the non-complex part $\restcomplex{\planargraph}$ is distributed like $\nocomplex(\tilde{n},\tilde{m})$.
\end{remark}

\subsection{Conditional random graphs}\label{sub:conditional_random_graphs}
Given a class $\cl$ of graphs it is sometimes quite difficult to directly analyse the random graph $\randomGraph=\randomGraph(n)\ur\cl(n)$. In such cases we will use the idea of conditional random graphs. Loosely speaking, we split $\cl$ into disjoint subclasses and consider for each subclass $\tilde{\cl}$ the random graph $\tilde{\randomGraph}=\tilde{\randomGraph}(n)\ur \tilde{\cl}(n)$, in other words, the random graph $\randomGraph$ conditioned on the event that $\randomGraph\in \tilde{\cl}$. If we can show that some graph property holds in all these \lq conditional\rq\ random graphs \whp, then \whp\ this property holds also in $\randomGraph$. The following definition and lemma make that idea more precise.
\begin{definition}\label{def:feasible}
	Given a class $\cl$ of graphs, a set $S$, and a function $\func:\cl\to S$, we call a sequence $\seq=(\term_n)_{n\in \N}$ {\em feasible} for $\left(\cl, \func\right)$ if for each $n \in \N$ there exists a graph $H \in \cl(n)$ such that $\func(H)=\term_n$. Moreover, for each $n\in \N$ we denote by $\left(\condGraph{\randomGraph}{\seq}\right)(n)$ a graph chosen uniformly at random from the set $\left\{H \in \cl(n): \func(H)=\term_n\right\}$. We will often omit the dependence on $n$ and write just $\condGraph{\randomGraph}{\seq}$ (i.e. \lq $\randomGraph$ conditioned on $\seq$\rq) instead of $\left(\condGraph{\randomGraph}{\seq}\right)(n)$.
\end{definition}
\begin{lem}[{\cite[Lemma 3.2]{cycles}}]\label{lem:conditional_random_graphs}
	Let $\cl$ be a class of graphs, $S$ a set, $\func:\cl\to S$ a function, $\property$ a graph property, and $\randomGraph=\randomGraph(n)\ur \cl(n)$. If for every sequence $\seq=(\term_n)_{n\in \N}$ that is feasible for $\left(\cl, \func\right)$ we have \whp\ $\condGraph{\randomGraph}{\seq} \in \property$, then we have \whp\ $\randomGraph \in \property$.
\end{lem}

\subsection{Internal structure of a random planar graph}\label{sub:internal_structure}
In the proofs of our main results we will use some results from \cite{planar,surface} on the internal structure of a random planar graph $\planargraph(n,m)$, e.g. asymptotic order of the core, which are reformulated to simplify asymptotic notation. 
\begin{thm}[\hspace{1sp}{\cite{planar,surface}}]\label{thm:internal_structure}
	Let $\planargraph=\planargraph(n,m)\ur \planarclass(n,m)$, $C=\core{\planargraph}$ be the core, $\Complexlargestcore=\complexlargestcore{\planargraph}$ the large complex part, $\Complexrestcore=\complexrestcore{\planargraph}$ the small complex part, $\Restcomplex=\restcomplex{\planargraph}$ the non-complex part, and $\Largestcomponent=\largestcomponent{\planargraph}$ the largest component of $\planargraph$. In addition, let $h=h(n)=\smallomega{1}$ be a function tending to $\infty$ arbitrarily slowly and $\funcL=\funcL(n)$ and $\funcR=\funcR(n)$ be as in \Cref{def:functionLR}. We assume that $m=m(n)\geq n/2+\smallomega{n^{2/3}}$ is such that $m\leq n+n^{1-\delta}$ for some $\delta>0$ and let $\beta:=\min\left\{\delta/2,1/5\right\}$. Then \whp\ the following hold.
\begin{enumerate}
\item\label{thm:internal_structure1}
$\maxdegree{C}=\Th{1}$;
\item\label{thm:internal_structure2}
$\numberVertices{\largestcomponent{C}}=\bigo{\funcL^{1-\beta}}$;
\item\label{thm:internal_structure3}
$\Complexlargestcore=\Largestcomponent$;
\item\label{thm:internal_structure4}
$\numberVertices{\Complexlargestcore}=\Th{\funcL}$;
\item\label{thm:internal_structure5}
$\numberVertices{\Complexrestcore}=\bigo{h\funcR^{2/3}}$;
\item\label{thm:internal_structure6}
$\numberVertices{\Restcomplex}=\Th{\funcR}$;
\item\label{thm:internal_structure7}
$\numberEdges{\Restcomplex}=\numberVertices{\Restcomplex}/2+\bigo{h\numberVertices{\Restcomplex}^{2/3}}$.
\end{enumerate}	
\end{thm}

\subsection{Properties of $\concentration{\nbins}{\nballs}$}
We will use the following basic properties of $\concentration{\nbins}{\nballs}$ introduced in \Cref{def:nu}. 
\begin{lem}\label{lem:nu}
Let the function $\concentration{\nbins}{\nballs}$ be defined as in \Cref{def:nu} and $\specialconcentration{\nbins}=\concentration{n}{n}$. Then we have
\begin{enumerate}
\item \label{lem:nu5} $\concentration{\nbins}{\nballs}>1$ for all $\nbins, \nballs \in \N$;
\item \label{lem:nu7} if $\nballs=\nballs(\nbins)=\bigo{n^{1/3}}$, then $\concentration{\nbins}{\nballs}\leq 5/3+\smallo{1}$;
\item \label{lem:nu1} if $\nballs=\nballs(\nbins)=\Th{\nbins}$, then $\concentration{\nbins}{\nballs}=\left(1+\smallo{1}\right)\log \nbins/\log\log\nbins$;
\item \label{lem:nu4} $\concentration{\nbins}{\nballs}$ is strictly increasing in the argument $\nballs$;
\item \label{lem:nu6}
if $\nballs=\nballs(\nbins)=\bigo{\nbins}$, then $\concentration{\nbins}{\nballs}=\smallomega{\nballs/\nbins}$;
\item \label{lem:nu3} if $\nballs=\nballs(\nbins)=\Th{\nbins}$ and $d=d(n)=\smallo{\nbins\left(\log \log \nbins\right)^2/\log \nbins}$, then $\concentration{\nbins}{\nballs+d}-\concentration{\nbins}{\nballs}=\smallo{1}$;
\item \label{lem:nu9} $\specialconcentration{\nbins}$ is strictly increasing;
\item \label{lem:nu2} if $c=c(\nbins)=\Th{1}$ and $\nballs=\nballs(\nbins)=\Th{\nbins}$, then $\concentration{c\nbins}{c\nballs}=\concentration{\nbins}{\nballs}+\smallo{1}$;
\item \label{lem:nu10} if $c=c(\nbins)=\Th{1}$, then $\concentration{\nbins}{c\nbins}-\specialconcentration{\nbins}=\left(\log c+\smallo{1}\right)\log n/\left(\log \log n\right)^2$.
\end{enumerate}
\end{lem}
We provide a proof of \Cref{lem:nu} in \Cref{sub:proof_nu}.

\section{Proof strategy}\label{sec:strategy}
In order to prove \Cref{thm:main_sub} on the two-point concentration of $\maxdegree{\planargraph(n,m)}$ when $m\leq n/2+\bigo{n^{2/3}}$, we will use the known fact that with positive probability the \ER\ random graph $G(n,m)$ is planar in this regime (see \Cref{thm:non_complex}). Thus, it suffices to determine $\maxdegree{G(n,m)}$ instead of $\maxdegree{P(n,m)}$, which we will do by proving that $\maxdegree{G(n,m)}$ \lq behaves\rq\ like the maximum load of an appropriate ball-into-bins model (see \Cref{sub:strategy_without_complex} for details).

The proof of \Cref{thm:main} will be based on the following result on the typical structure of $\planargraph=\planargraph(n,m)$, which can be derived by using statements from \cite{planar,surface}: Informally speaking, the largest component $\Largestcomponent=\largestcomponent{\planargraph}$ consists of a family $F$ of rooted tree components, which are connected via \lq few\rq\ edges between the {\em roots} of the tree components that are exactly the vertices of $\largestcomponent{\core{\planargraph}}$, i.e. the largest component of the core. The number of tree components in $F$ is much smaller than $\numberVertices{F}$, because the order of the core is typically much smaller than the order of the largest component (see \Cref{thm:internal_structure}\ref{thm:internal_structure2} and \ref{thm:internal_structure4}). In addition, the remaining part $\Rest=\rest{\planargraph}=\planargraph\setminus\Largestcomponent$ \lq behaves\rq\ like an \ER\ random graph $G(\tilde{n}, \tilde{n}/2)$ for a suitable $\tilde{n}=\tilde{n}(n)$. We refer to \Cref{fig:typical_structure} for an illustration of this structure. 

Then we will derive the two-point concentration of $\maxdegree{\Rest}$ by studying $G(\tilde{n}, \tilde{n}/2)$. Using the property that the number of tree components in $F$, and therefore also the number of roots, is small compared to $\numberVertices{F}$, we will show that the degrees of the roots are typically much smaller than $\maxdegree{F}$ (see \Cref{thm:forest_max_degree}\ref{thm:forest_max_degree2}). Together with the fact that the number of \lq additional\rq\ edges connecting the roots is \lq small\rq, this will yield $\maxdegree{L}=\maxdegree{F}$. Then the two-point concentration of $\maxdegree{L}$ will follow by analysing $\maxdegree{F}$ via the balls-into-bins model and Prüfer sequences (see \Cref{sec:forests}). In the following sections we will describe these ideas in more detail. In \Cref{sub:strategy_decomposition} we will use a graph decomposition and conditional random graphs to make the aforementioned structural result on $\planargraph$ more formal. Subsequently, we determine the maximum degrees of $F$ and $G(n,m)$ in \Cref{sub:strategy_complex_part,sub:strategy_without_complex}, respectively. 

\begin{figure}[t]
		\begin{tikzpicture}[scale=0.8, line width=0.6mm, every node/.style={circle,fill=lightgray, inner sep=0, minimum size=0.23cm}]

	\node (A) at (0,0) [draw,rectangle] {};
	\node (B) at (0.5,-0.95) [draw,rectangle] {};
	\node (C) at (1.5,-1.4) [draw,rectangle] {};
	\node (D) at (2.5,-0.95) [draw,rectangle] {};
	\node (E) at (3,0) [draw,rectangle] {};
	\node (F) at (2.5,0.95) [draw,rectangle] {};
	\node (G) at (1.5,1.4) [draw,rectangle] {};
	\node (H) at (0.5,0.95) [draw,rectangle] {};
	
	\node (I) at (4.2,0) [draw,rectangle] {};
	\node (J) at (5,1) [draw,rectangle] {};
	\node (K) at (6.2,0.7) [draw,rectangle] {};
	\node (L) at (6.2,-0.7) [draw,rectangle] {};
	\node (M) at (5,-1) [draw,rectangle] {};
		
	\node (N) at (5.8,-1.7) [draw] {};
	\node (O) at (5,-1.9) [draw] {};
	\node (P) at (6.6,-1.3) [draw] {};
	\node (Q) at (7.4,-1.1) [draw] {};
	\node (R) at (6.5,-2.3) [draw] {};
	\node (S) at (8.1,-0.7) [draw] {};
	\node (T) at (8.8,-0.4) [draw] {};
	\node (U) at (7.1,-1.8) [draw] {};
	\node (V) at (7.8,-1.8) [draw] {};
	\node (W) at (5.5,-2.7) [draw] {};
	\node (X) at (7.4,-2.5) [draw] {};
	\node (Y) at (8.5,-1.6) [draw] {};
	\node (Z) at (8.3,-2.3) [draw] {};
	\draw[-] (M) to (N);
	\draw[-] (M) to (O);
	\draw[-] (N) to (P);
	\draw[-] (P) to (Q);
	\draw[-] (N) to (R);
	\draw[-] (Q) to (S);
	\draw[-] (S) to (T);
	\draw[-] (P) to (U);
	\draw[-] (U) to (V);
	\draw[-] (U) to (X);
	\draw[-] (O) to (W);
	\draw[-] (V) to (Y);
	\draw[-] (V) to (Z);
	
	\node (1) at (2.7,-1.9) [draw] {};
	\node (2) at (3.3,-1.3) [draw] {};
	\node (3) at (3.3,-2.1) [draw] {};
	\node (4) at (4,-1.3) [draw] {};
	\node (5) at (3.9,-1.8) [draw] {};
	\node (6) at (4.5,-2.3) [draw] {};
	\draw[-] (D) to (1);
	\draw[-] (D) to (2);
	\draw[-] (2) to (3);
	\draw[-] (2) to (4);
	\draw[-] (2) to (5);
	\draw[-] (5) to (6);
	
	\node (7) at (7,1.1) [draw] {};
	\node (8) at (7,0.35) [draw] {};
	\node (9) at (7.7,1.1) [draw] {};
	\draw[-] (K) to (7);
	\draw[-] (K) to (8);
	\draw[-] (7) to (9);
	
	\node (10) at (0.6,-1.5) [draw] {};
	\node (11) at (-0.3,-1.5) [draw] {};
	\node (12) at (-1,-1.1) [draw] {};
	\node (13) at (-1,-1.8) [draw] {};
	\node (14) at (-1.7,-2) [draw] {};
	\node (15) at (-1.7,-1.5) [draw] {};
	\node (16) at (-1.7,-0.7) [draw] {};
	\node (17) at (-1.7,-1.1) [draw] {};
	\node (18) at (-2.2,-0.5) [draw] {};
	\node (19) at (-2.2,-1.1) [draw] {};
	
	\draw[-] (C) to (10);
	\draw[-] (10) to (11);
	\draw[-] (11) to (12);
	\draw[-] (11) to (13);
	\draw[-] (13) to (14);
	\draw[-] (12) to (15);
	\draw[-] (12) to (16);
	\draw[-] (12) to (17);
	\draw[-] (16) to (18);
	\draw[-] (17) to (19);
	
	\node (20) at (4.4,1.2) [draw] {};
	\node (21) at (5.6,1.2) [draw] {};
	\draw[-] (J) to (20);
	\draw[-] (J) to (21);
	
	\node (22) at (-0.7,0.3) [draw] {};
	\node (23) at (-1.3,0.5) [draw] {};
	\draw[-] (A) to (22);
	\draw[-] (22) to (23);
	
	\node (24) at (3.2,1.2) [draw] {};
	\draw[-] (F) to (24);
		
	\draw[-,thin] (A) to (B);
	\draw[-,thin] (B) to (C);
	\draw[-,thin] (C) to (D);
	\draw[-,thin] (D) to (E);
	\draw[-,thin] (E) to (F);
	\draw[-,thin] (F) to (G);
	\draw[-,thin] (G) to (H);
	\draw[-,thin] (H) to (A);
	\draw[-,thin] (C) to (H);
	\draw[-,thin] (I) to (J);
	\draw[-,thin] (J) to (K);
	\draw[-,thin] (K) to (L);
	\draw[-,thin] (L) to (M);
	\draw[-,thin] (M) to (I);
	\draw[-,thin] (E) to (I);
	
	\draw[-, thin] (-2.5,-3) to (-2.5,-3.1);
	\draw[-, thin] (-2.5,-3.1) to (9,-3.1);
	\draw[-, thin] (9,-3) to (9,-3.1);
	\node () at (3.75,-3.5) [draw=none, rectangle, fill=white] {$L$};
	
	\node (25) at (11.5,1.4) [draw] {};
	\node (26) at (11.5,0.8) [draw] {};
	\node (27) at (11.5,0.2) [draw] {};
	\node (28) at (11.8,-0.4) [draw] {};
	\node (29) at (11.1,-0.4) [draw] {};
	\node (30) at (11.1,-1) [draw] {};
	\node (31) at (10.8,-1.6) [draw] {};
	\node (32) at (11.4,-1.6) [draw] {};
	
	\draw[-] (25) to (26);
	\draw[-] (26) to (27);
	\draw[-] (27) to (28);
	\draw[-] (27) to (29);
	\draw[-] (29) to (30);
	\draw[-] (30) to (31);
	\draw[-] (30) to (32);
	
	\node (33) at (13,1.4) [draw] {};
	\node (34) at (12.6,0.8) [draw] {};
	\node (35) at (13.4,0.8) [draw] {};
	\node (36) at (13.1,0.2) [draw] {};
	\node (37) at (13.7,0.2) [draw] {};
	\node (38) at (13.1,-0.4) [draw] {};
	\draw[-] (33) to (34);
	\draw[-] (33) to (35);
	\draw[-] (35) to (36);
	\draw[-] (35) to (37);
	\draw[-] (36) to (38);
	
	\node (39) at (12.5,-1.2) [draw] {};
	\node (40) at (13.1,-1.2) [draw] {};
	\node (41) at (13.7,-0.8) [draw] {};
	\node (42) at (13.7,-1.6) [draw] {};
	\draw[-] (39) to (40);
	\draw[-] (40) to (41);
	\draw[-] (40) to (42);
	
	\node (43) at (14.4,1.3) [draw] {};
	\node (44) at (15.1,1.3) [draw] {};
	\draw[-] (43) to (44);
	
	\node (45) at (14.4,0.6) [draw] {};
	\node (46) at (15.1,0.6) [draw] {};
	\draw[-] (45) to (46);
	
	\node (47) at (14.4,-0.1) [draw] {};
	\node (48) at (15.1,-0.1) [draw] {};
	\draw[-] (47) to (48);
	
	\node (49) at (14.4,-0.8) [draw] {};
	\node (50) at (15.1,-0.8) [draw] {};
	\draw[-] (49) to (50);
	
	\node (51) at (14.4,-1.5) [draw] {};
	\node (52) at (15.1,-1.5) [draw] {};
	\draw[-] (51) to (52);
	
	\node (53) at (16.1,1.4) [draw] {};
	\node (54) at (16.9,1.4) [draw] {};
	\node (55) at (16.1,0.8) [draw] {};
	\node (56) at (16.9,0.8) [draw] {};
	\node (57) at (16.1,0.2) [draw] {};
	\node (58) at (16.9,0.2) [draw] {};
	\node (59) at (16.1,-0.4) [draw] {};
	\node (60) at (16.9,-0.4) [draw] {};
	\node (61) at (16.1,-1) [draw] {};
	\node (62) at (16.9,-1) [draw] {};
	\node (63) at (16.1,-1.6) [draw] {};
	\node (64) at (16.9,-1.6) [draw] {};
	
	\draw[-, thin] (10.5,-2.1) to (10.5,-2.2);
	\draw[-, thin] (10.5,-2.2) to (17.2,-2.2);
	\draw[-, thin] (17.2,-2.1) to (17.2,-2.2);
	\node () at (13.75,-2.7) [draw=none, fill=white, rectangle] {$R$};	
\end{tikzpicture}
\caption{Typical structure of $\planargraph=\planargraph(n,m)$ when $m$ is as in \Cref{thm:main}: The largest component $\Largestcomponent=\largestcomponent{\planargraph}$ consists of a family of rooted tree components, which are connected via \lq few\rq\ edges (drawn with thin lines) between the roots (square boxes). The remaining part $\Rest=\planargraph\setminus\Largestcomponent$ \lq behaves\rq\ like an \ER\ random graph $G(\tilde{n}, \tilde{n}/2)$ for a suitable $\tilde{n}=\tilde{n}(n)$.}
\label{fig:typical_structure}
\end{figure}
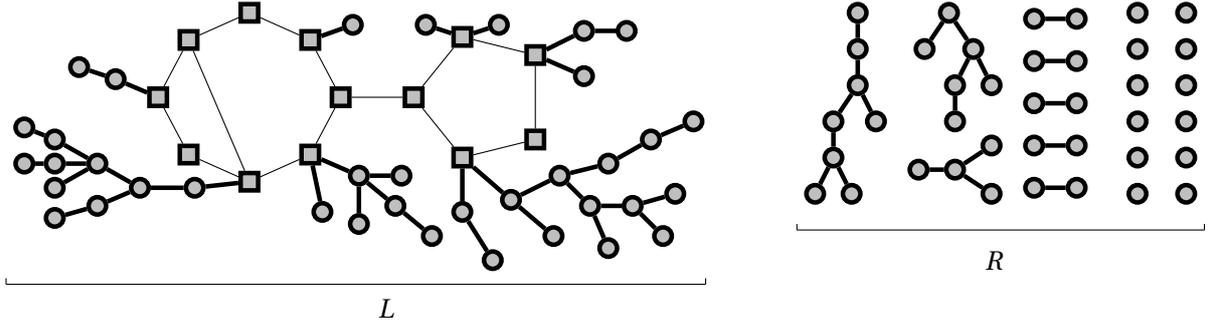

\subsection{Decomposition and conditional random graphs}\label{sub:strategy_decomposition}
Instead of considering $\Largestcomponent$ and $\Rest$ directly, we will actually split the random planar graph $\planargraph$ into the large complex part $\Complexlargestcore=\complexlargestcore{\planargraph}$, the small complex part $\Complexrestcore=\complexrestcore{\planargraph}$, and the non-complex part $\Restcomplex=\restcomplex{\planargraph}$ (see \Cref{sub:decomposition} for a formal definition of $\Complexlargestcore$, $\Complexrestcore$, and $\Restcomplex$). We then use the fact that by \Cref{thm:internal_structure}\ref{thm:internal_structure3} we have \whp\
\begin{align}\label{eq:37}
\Largestcomponent=\Complexlargestcore,
\end{align}
which also implies that \whp\ $\Rest=\Complexrestcore \cup \Restcomplex$. In order to analyse $\Complexlargestcore$, $\Complexrestcore$, and $\Restcomplex$, we will use the concept of conditional random graphs (see \Cref{sub:conditional_random_graphs}): For given $\lambda, \sigma\in \N$ and a core $C$, we denote by $\tilde{\planargraph}$ the random planar graph $\planargraph$ conditioned on the event that $\core{\planargraph}=C$, $\numberVertices{\complexlargestcore{\planargraph}}=\lambda$, and $\numberVertices{\complexrestcore{\planargraph}}=\sigma$. By \Cref{rem:connection_conditional} we have
\begin{align}
\complexlargestcore{\tilde{\planargraph}}&\sim\complexgraph\left(\largestcomponent{C},\lambda\right),\label{eq:25}\\
\complexrestcore{\tilde{\planargraph}}&\sim\complexgraph\left(\rest{C},\sigma\right),\label{eq:26}\\
\restcomplex{\tilde{\planargraph}}&\sim\nocomplex(u,w),\label{eq:27}
\end{align}
where the random graphs on the right hand side are as defined in \Cref{def:random_complex_part,def:nocomplex}, $\largestcomponent{C}$ the largest component of $C$, $\rest{C}=C\setminus\largestcomponent{C}$, $u:=n-\lambda-\sigma$, and $w:=m-\numberEdges{C}+\numberVertices{C}-\lambda-\sigma$. 

Roughly speaking, there is the following elementary but useful relation between the \lq conditional\rq\ random graph $\tilde{\planargraph}$ and the original random planar graph $\planargraph$ (see \Cref{lem:conditional_random_graphs}): If for all \lq typical\rq\ choices of $C$, $\lambda$, and $\sigma$ \whp\ a graph property holds in $\tilde{\planargraph}$, then \whp\ this property holds in $\planargraph$. In order to determine what \lq typical\rq\ choices of $C$, $\lambda$, and $\sigma$ are, we use known results on the internal structure of $\planargraph$ (see \Cref{thm:internal_structure}). For example, if we know that \whp\ the core $\core{\planargraph}$ satisfies a certain structure, e.g. the maximum degree is bounded or the number of vertices lies in a certain interval, then typical choices of $C$ are those cores having this structure.

Due to this relation between $\planargraph$ and $\tilde{\planargraph}$ and (\ref{eq:25})--(\ref{eq:27}) it suffices to consider the random graphs $\complexgraph\left(C,q\right)$ and $\nocomplex(n,m)$ for fixed values of $C$, $q$, $n$, and $m$. We will see that if we consider $\nocomplex(n,m)$, then we always have $m=n/2+\bigo{n^{2/3}}$. It is well known that in this regime the \ER\ random graph $G(n,m)$ has with positive probability no complex components (see \Cref{thm:non_complex}). Hence, we can consider $\maxdegree{G(n,m)}$ instead of $\maxdegree{\nocomplex(n,m)}$, which we will do in \Cref{sub:strategy_without_complex}. Furthermore, in \Cref{sub:strategy_complex_part} we will study $\complexgraph\left(C,q\right)$ by using the balls-into-bins model.

We emphasize that the decomposition $\planargraph=\Complexlargestcore~\dot{\cup}~\Complexrestcore~\dot{\cup}~\Restcomplex$ describes the structure of $\planargraph$ as stated at the beginning of \Cref{sec:strategy} and illustrated in \Cref{fig:typical_structure}: By (\ref{eq:37}) the large complex part $\Complexlargestcore$ corresponds to the largest component $\Largestcomponent=\largestcomponent{\planargraph}$. Using (\ref{eq:25}) this implies that the asymptotic behaviour of $\Largestcomponent$ can be deduced by that of $\complexgraph\left(C,q\right)$ for a suitable core $C$ and $q\in \N$. The random graph $\complexgraph\left(C,q\right)$ can be constructed by replacing each vertex of $C$ randomly by a rooted tree component such that a complex graph with $q$ vertices is obtained. Furthermore, in our applications $\maxdegree{C}$ will be bounded and $\numberVertices{C}$ will be \lq small\rq\ compared to $q$ (see \Cref{thm:internal_structure}\ref{thm:internal_structure1} and \ref{thm:internal_structure2}). This implies that $\complexgraph\left(C,q\right)$, and therefore also $\Largestcomponent$, consists of a family of rooted tree components (containing the edges not lying in $C$), which are connected via \lq few\rq\ additional edges (which are the edges lying in $C$). For the structure of the remaining part $\Rest=\planargraph\setminus\Largestcomponent$ we observe that $\Rest$ corresponds to $\Complexrestcore \cup \Restcomplex$ (see (\ref{eq:37})). Combining the facts that $\numberVertices{\Complexrestcore}$ will be \lq small\rq\ compared to $\numberVertices{\Restcomplex}$ and $\numberEdges{\Restcomplex}\approx \numberVertices{\Restcomplex}/2$ (see \Cref{thm:internal_structure}\ref{thm:internal_structure5} and \ref{thm:internal_structure7}) with (\ref{eq:27}), we obtain that $\Rest$ is closely related to $\nocomplex(\tilde{n},\tilde{n}/2)$, and therefore also to $G(\tilde{n},\tilde{n}/2)$, for a suitable $\tilde{n}\in\N$.

\subsection{Random complex part and forests with specified roots}\label{sub:strategy_complex_part}
Let $C$ be a core (on vertex set $[\numberVertices{C}]$) and $q\in \N$. In \Cref{def:random_complex_part} we denoted by $\complexgraph\left(C,q\right)$ a graph chosen uniformly at random from the family $\complexclass(C,q)$ of all complex graphs with core $C$ and vertex set $[q]$. Moreover, we let $\forestclass(n, \ntrees)$ be the class of forests on vertex set $[n]$ consisting of $\ntrees$ tree components such that each vertex from $[\ntrees]$ lies in a different tree component. The elements in $\forestclass(n, \ntrees)$ are called {\em forests with specified roots} and the vertices in $[\ntrees]$ {\em roots}. For simplicity, we will often just write forest instead of forest with specified roots. We can construct $\complexgraph=\complexgraph\left(C,q\right)$ by choosing a random forest $\forest=\forest(q, \numberVertices{C})\ur\forestclass(q, \numberVertices{C})$ and replacing each vertex $v$ in $C$ by the tree component with root $v$. For the degrees of vertices in $\complexgraph$ we obtain $\degree{v}{\complexgraph}=\degree{v}{C}+\degree{v}{\forest}$ for $v\in C$ and $\degree{v}{\complexgraph}=\degree{v}{\forest}$ otherwise. In our applications we will have that $\maxdegree{C}$ is bounded and $\numberVertices{C}=\bigo{q^{1-\beta}}$ for some $\beta>0$, i.e. $\numberVertices{C}$ is \lq small\rq\ compared to $q$ (see \Cref{thm:internal_structure}\ref{thm:internal_structure1}, \ref{thm:internal_structure2}, and \ref{thm:internal_structure4}). This will imply that \whp\ $\maxdegree{\complexgraph}=\maxdegree{\forest}$ (see \Cref{thm:random_complex_part}). 

In order to determine $\maxdegree{\forest}$, we will introduce a bijection between $\forestclass(n, \ntrees)$ and $\sequences{n}{\ntrees}:=[n]^{n-\ntrees-1}\times [\ntrees]$ similar to Prüfer sequences for trees (see \Cref{sub:pruefer}). Given a forest $\forest\in \forestclass(n,\ntrees)$ we recursively delete the leaf, i.e. a vertex with degree one, with largest label and thereby build a sequence by noting the unique neighbours of the leaves. We will show in \Cref{thm:pruefer} that this is indeed a bijection and that the degree of a vertex $v$ is determined by the number of occurrences of $v$ in the sequence (see (\ref{eq:12})). It is straightforward to construct a random element from $\sequences{n}{\ntrees}$ by a balls-into-bins model such that the load of a bin equals the number of occurrences in the sequence of the corresponding element. Thus, we will derive the concentration of the maximum degree $\maxdegree{\forest}$ from a concentration result on the maximum load. We refer to \Cref{fig:construction_complex_part} for an illustration of the construction of $\complexgraph\left(C,q\right)$ via the random forest $\forest$ and the balls-into-bins model.

\subsection{\ER\ random graph and the balls-into-bins model}\label{sub:strategy_without_complex}
Given $n$ bins $\bin_1, \ldots, \bin_\nbins$ and $2m$ balls $\ball_1, \ldots, \ball_{2m}$ we denote by $\locationBit_i$ the index of the bin to which the $i$-th ball $\ball_i$ is assigned for each $i\in[2m]$. We will consider the random multigraph $\multigraph$ with $\vertexSet{\multigraph}=[n]$ and $\edgeSet{\multigraph}=\setbuilder{\left\{\locationBit_{2i-1}, \locationBit_{2i}\right\}}{i \in [m]}$ (see also \Cref{fig:Gnm_balls_bins}). We will see that conditioned on $\multigraph$ being simple, $\multigraph$ is distributed like $G(n,m)$. Furthermore, we will show that as long as $m=\bigo{n}$, with positive probability $\multigraph$ is simple. Hence, the concentration of $\maxdegree{G(n,m)}$ will follow by the concentration of the maximum load of a bin (see \Cref{thm:concentration_balls_bins}).

\begin{figure}[t]
	\begin{tikzpicture}[scale=1.12, very thick, every node/.style={circle,draw,fill=lightgray, minimum size=0.5cm, inner sep=0pt}]	
		\draw[-] (0,2) to (0,0);
		\draw[-] (-0.02,0) to (1.22,0);
		\draw[-] (1.2,0) to (1.2,2);
		\node () at (0.6,-0.35) [rectangle] {\large \textbf{1}};
		
		\draw[-] (1.8,2) to (1.8,0);
		\draw[-] (1.78,0) to (3.02,0);
		\draw[-] (3,0) to (3,2);
		\node () at (2.4,-0.35) [rectangle] {\large \textbf{2}};
		
		\draw[-] (3.6,2) to (3.6,0);
		\draw[-] (3.58,0) to (4.82,0);
		\draw[-] (4.8,0) to (4.8,2);
		\node () at (4.2,-0.35) [rectangle] {\large \textbf{3}};
		
		\draw[-] (5.4,2) to (5.4,0);
		\draw[-] (5.38,0) to (6.62,0);
		\draw[-] (6.6,0) to (6.6,2);
		\node () at (6,-0.35) [rectangle] {\large \textbf{4}};
		
		\draw[-] (7.2,2) to (7.2,0);
		\draw[-] (7.18,0) to (8.42,0);
		\draw[-] (8.4,0) to (8.4,2);
		\node () at (7.8,-0.35) [rectangle] {\large \textbf{5}};
		
		\node (1) at (11.5,-0.2) [rectangle] {\large \textbf{1}};
		\node (2) at (12.53,0.52) [rectangle] {\large \textbf{2}};
		\node (3) at (12.14,1.7) [rectangle] {\large \textbf{3}};
		\node (4) at (10.86,1.7) [rectangle] {\large \textbf{4}};
		\node (5) at (10.47,0.52) [rectangle] {\large \textbf{5}};
		
		\node () at (4.6,-1) [rectangle,fill=none,draw=none] {(a) balls-into-bins experiment};
		\node () at (11.5,-1) [rectangle,fill=none,draw=none] {(b) multigraph $M$};
		
		\draw[-,dashed] (9.5,-1.1) to (9.5,2.15);

		\node () at (7.8,0.75) [] {\textbf{1}};
		\node () at (3.92,0.3) [] {\textbf{2}};
		\node () at (7.52,0.3) [dashed] {\textbf{3}};
		\node () at (0.6,0.3) [dashed] {\textbf{4}};
		
		\node () at (2.68,0.3) [dotted] {\textbf{5}};
		\node () at (8.08,0.3) [dotted] {\textbf{6}};
		\node () at (2.12,0.3) [loosely dotted] {\textbf{7}};
		\node () at (4.48,0.3) [loosely dotted] {\textbf{8}};
		\draw[-,dotted] (2) to (5);	
		\draw[-,dashed] (1) to (5);
		\draw[-] (3) to (5);
		\draw[-,loosely dotted] (2) to (3);
				
	\end{tikzpicture}
	\caption{Construction of a multigraph $M$ via the balls-into-bins model. In (a) we have $n=5$ bins, $2m=8$ balls, and $A_1=5, A_2=3, A_3=5, \ldots, A_8=3$. In (b) this results in a multigraph $M$ with $\vertexSet{M}=[5]$ and $\edgeSet{M}=\left\{\{5,3\}, \{5,1\}, \{2,5\}, \{2,3\}\right\}$. The maximum load of a bin ($=3$) corresponds to the maximum degree $\maxdegree{M}=3$ of $M$.}
	\label{fig:Gnm_balls_bins}
\end{figure}
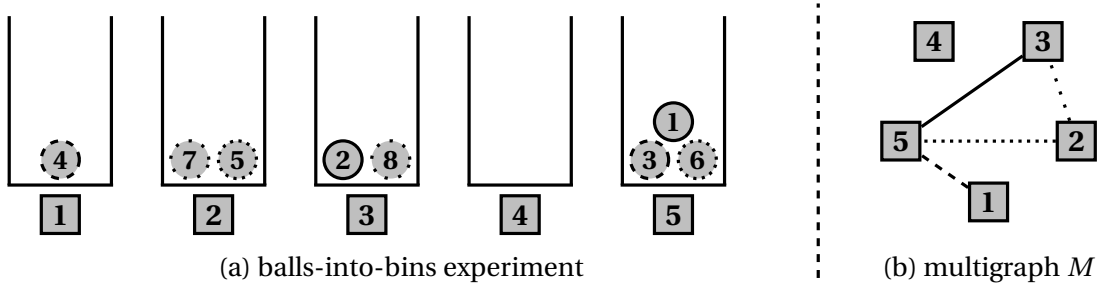

\subsection{Double counting}\label{sub:strategy_dense}
To prove \Cref{thm:main_dense}, we will combine results on the asymptotic number of planar graphs from \cite{gim} (see \Cref{thm:number_planar}) and a double counting argument (see \Cref{lem:isolated}) and deduce that for all fixed $\niv,\nie\in\N$ we have 
\begin{align}\label{eq:33}
	\liminf_{n\to \infty}~ \prob{P \text{ has }\niv \text{ isolated vertices and }\nie \text{ isolated edges}}>0,
\end{align}
where we call a vertex isolated if it has degree zero and say that an edge is isolated if both endpoints have degree one. Then we introduce an operation that uses an isolated vertex and two isolated edges to increase the maximum degree of a graph by one (see \Cref{fig:double_counting}). Starting with a graph that has \lq many\rq\ isolated vertices and isolated edges, we can repeatedly apply this operation to create lots of graphs with various maximum degrees (see \Cref{lem:increase_max_degree}). Together with (\ref{eq:33}) this will imply that also $\maxdegree{\planargraph}$ takes \lq many\rq\ different values.

\section{Balls into bins}\label{sec:balls_bins}
Balls-into-bins models have been extensively studied in the literature (see e.g. \cite{johnson_kotz, mitzenmacher_upfal}). Throughout the paper, we will use the following model. Given $\nbins$ bins $\bin_1, \ldots, \bin_\nbins$ we sequentially assign $\nballs$ balls $\ball_1, \ldots, \ball_\nballs$ to those $\nbins$ bins by choosing a bin for each ball, independently and uniformly at random. Let $\location=\left(\locationBit_1, \ldots, \locationBit_\nballs\right)$ be the location vector, i.e. $\locationBit_i$ is the index of the bin to which the $i$-th ball $\ball_i$ is assigned. For each $j\in [\nbins]$ we call the number of balls in the $j$-th bin $\bin_j$ the {\em load} $\load_j=\load_j(\location)$. We write $\loadvector=\loadvector(\location)=\left(\load_1, \ldots, \load_\nbins\right)$ for the vector of all loads and denote by $\maxload=\maxload(\location)=\max_{j \in [\nbins]}\load_j$ the maximum load in a bin. For $\ntrees\in[\nbins]$ we let $\maxload_{ \ntrees}=\maxload_{\ntrees}(\location)=\max_{j \in [\ntrees]}\load_j$ be the maximum load among the first $\ntrees$ bins $\bin_1, \ldots, \bin_\ntrees$. We write $\binsandballs{\nbins}{\nballs}$ for a random vector distributed like the location vector $\location$ of a balls-into-bins experiment with $\nbins$ bins and $\nballs$ balls, denoted by 
\begin{align*}
\location\sim \binsandballs{\nbins}{\nballs}
\end{align*}
and $\maxbinsandballs{\nbins}{\nballs}$ for a random variable distributed like the maximum load $\maxload$, which we denote by
\begin{align*}
\maxload=\maxload(\location)\sim \maxbinsandballs{\nbins}{\nballs}.
\end{align*}

Gonnet \cite{gonnet} proved in the case $\nbins=\nballs$ that \whp\ $\maxbinsandballs{\nbins}{\nbins}=\left(1+\smallo{1}\right)\log n/\log\log n$. Later Raab and Steger \cite{balls_bins1} considered $\maxbinsandballs{\nbins}{\nballs}$ for different ranges of $\nballs$. Amongst other results, they showed that \whp\ $\maxbinsandballs{\nbins}{\nballs}=\left(1+\smallo{1}\right)\log n/\log\log n$ is still true, as long as $\nballs=\Th{\nbins}$. In the following we refine their result, showing that if $\nballs=\bigo{\nbins}$, then \whp\ $\maxbinsandballs{\nbins}{\nballs}$ is actually concentrated at two values.

Before proving that rigorously, we motivate this result by providing the following heuristic. For $l=l(n)\in\N$ we let $X^{(l)}$ be the number of bins with load $l$. We have 
\begin{align}\label{eq:28}
	\expec{X^{(l)}}=\nbins\binom{\nballs}{l}\left(1/n\right)^l\left(1-1/n\right)^{\nballs-l}=:\mu(l).
\end{align}
We expect that the load $l$ of a bin is much smaller than $\nballs$ and therefore we have
\begin{align*}
	\mu(l)=\Th{1}k^le^ll^{-l-1/2}n^{-l+1}.
\end{align*}
Intuitively, the maximum load $\maxload\sim \maxbinsandballs{\nbins}{\nballs}$ should be close to the largest $l$ for which $\mu(l)=\Th{1}$ is satisfied, in other words, $\log \left(\mu(\maxload)\right)$ should be close to 0. This motivates the definition of $\concentration{\nbins}{\nballs}$ in \Cref{def:nu} as the unique positive zero of the function
\begin{align*}
	f(l)=f_{\nbins, \nballs}(l):=l\log k+l-(l+1/2)\log l-(l-1)\log n,
\end{align*}
which is asymptotically equal to $\log\left(\mu(l)\right)$ up to an additive constant. We will use the first and second moment method (see e.g. \cite{probmethod,rg1}) to make that heuristic rigorous and show that the maximum load $\maxload\sim \maxbinsandballs{\nbins}{\nballs}$ is strongly concentrated around $\concentration{\nbins}{\nballs}$.

\begin{thm}\label{thm:concentration_balls_bins}
If $\nballs=\nballs(\nbins)=\bigo{\nbins}$ and $\varepsilon>0$, then \whp
\begin{align*}
\rounddown{\concentration{\nbins}{\nballs}-\varepsilon}\lessorequal \maxbinsandballs{\nbins}{\nballs}\lessorequal \rounddown{\concentration{\nbins}{\nballs}+\varepsilon}.
\end{align*}
\end{thm}
\begin{proof}
Let $\location\sim \binsandballs{\nbins}{\nballs}$ be the location vector, $\load_j=\load_j(\location)$ the load of bin $\bin_j$ for each $j\in[\nbins]$, and $\maxload=\maxload(\location)\sim \maxbinsandballs{\nbins}{\nballs}$ the maximum load. First we consider the case $\nballs\leq \nbins^{1/3}$. Then we have
\begin{align}\label{eq:18}
	\prob{\maxload=1}=\prod_{i=1}^{\nballs-1}\left(1-\frac{i}{\nbins}\right)\greaterorequal \left(1-\frac{\nballs}{\nbins}\right)^\nballs=1-\smallo{1}.
\end{align}
Due to \Cref{lem:nu}\ref{lem:nu5} and \ref{lem:nu7} we have $1<\concentration{\nbins}{\nballs}\leq 7/4$ for $\nbins$ large enough. Together with (\ref{eq:18}) this shows the statement for the case $\nballs\leq \nbins^{1/3}$. Hence, it remains to consider the case $\nballs>\nbins^{1/3}$. For $l\in[\nballs]$ and $j\in[\nbins]$ we let $X_j^{(l)}=1$ if $\load_j=l$, i.e. the number $\load_j$ of balls (among $\nballs$ balls) in the $j$-th bin $\bin_j$ is equal to $l$, and $X_j^{(l)}=0$ otherwise. In addition, we let $X^{(l)}=\sum_{j=1}^{\nbins}X_j^{(l)}$ be the number of bins with load $l$. Then we have 
$\prob{X_j^{(l)}=1}=\binom{\nballs}{l}\left(1/n\right)^l\left(1-1/n\right)^{\nballs-l}$ and obtain (\ref{eq:28}). If $l=\bigo{\nballs^{1/2}}$, then
$\binom{\nballs}{l}=\Th{1}\nballs^le^l/l^{l+1/2}$,
where we used Stirling's formula for $l!$. Hence, we get 
\begin{align}\label{eq:13}
	\mu(l)=\Th{1}\frac{\nballs^l e^l}{l^{l+1/2}n^{l-1}},
\end{align}
because $\left(1-1/n\right)^{\nballs-l}=\Theta(1)$. For an upper bound of the maximum load $\maxload$ we will use the first moment method. Let $l^\ast=l^\ast(\nbins):=\rounddown{\concentration{\nbins}{\nballs}+\varepsilon}+1$ and $\tau=\tau(n):=l^\ast-\concentration{\nbins}{\nballs}\geq\varepsilon$. Due to \Cref{lem:nu}\ref{lem:nu1} and \ref{lem:nu4} and the assumption $\nballs> \nbins^{1/3}$ we have $l^\ast=\bigo{\nballs^{1/2}}$. Thus, equation (\ref{eq:13}) holds for $l=l^\ast$ and by the definition of $\Concentration=\concentration{\nbins}{\nballs}$ we obtain
\begin{align}
	\mu\left(l^\ast\right)=\Th{1}\frac{\nballs^{\Concentration+\tau}e^{\Concentration+\tau}}{\left(\Concentration+\tau\right)^{\Concentration+\tau+1/2}n^{\Concentration+\tau-1}}=\Th{1}\left(\frac{\nballs e}{\nbins\left(\Concentration+\tau\right)}\right)^\tau\left(\frac{\Concentration}{\Concentration+\tau}\right)^{\Concentration+1/2}.
\end{align}
Together with \Cref{lem:nu}\ref{lem:nu6} this yields $\mu\left(l^\ast\right)=\smallo{1}$. Due to \Cref{lem:nu}\ref{lem:nu6} we have $\mu\left(l+1\right)/\mu\left(l\right)=\left(\nballs-l\right)/\left(\left(l+1\right)\left(\nbins-1\right)\right)=o(1)$ for all $l\geq l^\ast$. Hence,
\begin{align*}
	\prob{\maxload\geq l^\ast}\lessorequal \sum_{l\geq l^\ast}\mu(l)=\left(1+\smallo{1}\right)\mu(l^\ast)=o(1).
\end{align*}
For a lower bound, we will show that $\prob{X^{\left(l_\ast\right)}>0}=1-o(1)$, where $l_\ast=l_\ast(\nbins):=\rounddown{\concentration{\nbins}{\nballs}-\varepsilon}$, using the second moment method. In the following we consider the random variables $X_j^{(l)}$ and $X^{(l)}$ only for $l=l_\ast$ and therefore we use $X_j=X_j^{\left(l_\ast\right)}$ and $X=X^{\left(l_\ast\right)}$ for simplicity. In order to apply the second moment method, we will show $\expec{X}=\smallomega{1}$ and $\expec{X_i X_j}=\left(1+\smallo{1}\right)\expec{X_i}\expec{X_j}$ for all $i\neq j$. We let $\rho=\rho(n):=\Concentration-l_\ast\geq\varepsilon$ and by (\ref{eq:13}), \Cref{lem:nu}\ref{lem:nu6}, and the definition of $\Concentration$ we obtain
\begin{align*}
	\mu\left(l_\ast\right)=\Th{1}\frac{\nballs^{\Concentration-\rho}e^{\Concentration-\rho}}{\left(\Concentration-\rho\right)^{\Concentration-\rho+1/2}n^{\Concentration-\rho-1}}=\Th{1}\left(\frac{\nbins\left(\Concentration-\rho\right)}{\nballs e}\right)^\rho\left(\frac{\Concentration}{\Concentration-\rho}\right)^{\Concentration+1/2}=\smallomega{1}.
\end{align*}
Next, we note that conditioned on the event $X_i=1$, i.e. $\load_i=l_\ast$, the loads $\load_j$ for $j\neq i$ are distributed like the loads of a balls-into-bins experiment with $\nbins-1$ bins and $\nballs-l_\ast$ balls, and thus
\begin{align*}
	\condprob{X_j=1}{X_i=1}\equal\binom{\nballs-l_\ast}{l_\ast}\left(1/(\nbins-1)\right)^{l_\ast}\left(1-1/(\nbins-1)\right)^{\nballs-2l_\ast}.
\end{align*}
Hence, we obtain
\begin{align*}
	\frac{\expec{X_i X_j}}{\expec{X_i}\expec{X_j}}\equal\frac{\condprob{X_j=1}{X_i=1}}{\prob{X_j=1}}\equal\frac{\binom{\nballs-l_\ast}{l_\ast}\left(1/(\nbins-1)\right)^{l_\ast}\left(1-1/(\nbins-1)\right)^{\nballs-2l_\ast}}{\binom{\nballs}{l_\ast}\left(1/n\right)^{l_\ast}\left(1-1/n\right)^{\nballs-l_\ast}} \equal1+\smallo{1},
\end{align*}
where we used the assumption $\nballs>\nbins^{1/3}$ and the fact $l_\ast=\bigo{\log n/\log \log n}$ due to \Cref{lem:nu}\ref{lem:nu1} and \ref{lem:nu4}. Thus, by the second moment method we obtain $\prob{X>0}=1-o(1)$, which finishes the proof.
\end{proof}

Next, we show that if we consider a \lq small\rq\ subset of bins, then the maximum load in one of these bins is significantly smaller than the maximum load of all bins. We will use this fact later when we relate random forests to the balls-into-bins model (see \Cref{sec:forests}), in which this \lq small\rq\ subset will correspond to the set of all roots.

\begin{lem}\label{lem:max_load_subset}
Let $\nballs=\nballs(n)$ and $t=t(n) \in \N$ be such that $\nballs=\Th{n}$ and $t= \bigo{n^{1-\beta}}$ for some $\beta>0$. Let $\location\sim\binsandballs{\nbins}{\nballs}$, $\maxload=\maxload(\location)\sim \maxbinsandballs{\nbins}{\nballs}$ be the maximum load, and $\maxload_{t}=\maxload_{t}(\location)$ be the maximum load in one of the first $t$ bins. Then, \whp\ $\maxload-\maxload_{ t}\equal\smallomega{1}$.
\end{lem}

\begin{proof}
We observe that $\maxload-\maxload_{ t}$ is strictly decreasing in $t$. Thus, it suffices to show $\maxload-\maxload_{ t}\equal\smallomega{1}$ for $t=\rounddown{n^{1-\beta}}$ and $\beta\in (0,1)$. We denote by $S_t$ the total number of balls in the first $t$ bins. We have $\expec{S_t}=t\nballs/\nbins$ and $\variance{S_t}\leq \expec{S_t}$. Hence, by Chebyshev's inequality, we have \whp\
\begin{align}\label{eq:19}
	S_t\lessorequal \frac{t\nballs}{\nbins}+\left(\frac{t\nballs}{\nbins}\right)^{2/3}=:\bar{l}=\bar{l}(n).
\end{align}
Conditioned on the event $S_t=l$ for $l\in\N$, $\maxload_{t}$ is distributed like $\maxbinsandballs{t}{l}$. Thus,
\begin{align*}
	\prob{\maxload_{t}\lessorequal \rounddown{\concentration{t}{\bar{l}}}+1}&\greaterorequal \sum_{l=1}^{\bar{l}}\probLarge{S_t=l}\condprob{\maxload_{t}\lessorequal \rounddown{\concentration{t}{\bar{l}}}+1}{S_t=l}
	\\
	&\greaterorequal \prob{S_t\leq \bar{l}}\prob{\maxbinsandballs{t}{\bar{l}}\lessorequal \rounddown{\concentration{t}{\bar{l}}}+1}
	\equal\left(1-\smallo{1}\right),
\end{align*}
where the last equality follows from \Cref{thm:concentration_balls_bins} and (\ref{eq:19}). Due to \Cref{lem:nu}\ref{lem:nu1} and the assumption $t=\rounddown{n^{1-\beta}}$ we get $\concentration{t}{\bar{l}}=\left(1+\smallo{1}\right)\log t/\log \log t=\left(1-\beta+\smallo{1}\right)\log n/\log \log n$, which yields \whp\ $\maxload_{t}\leq\left(1-\beta+\smallo{1}\right)\log n/\log\log n$. By \Cref{lem:nu}\ref{lem:nu1} we have \whp\ $\maxload=\left(1+\smallo{1}\right)\log n/\log\log n$. Hence, we obtain \whp\ $\maxload-\maxload_{t}\greaterorequal \left(\beta+\smallo{1}\right)\log n/\log\log n=\smallomega{1}$, as desired.
\end{proof}

\section{\ER\ random graph and graphs without complex components}\label{sec:ergraph}
We start this section by providing a relation between the degree sequence of the \ER\ random graph $G(n,m)$ and the loads of a balls-into-bins model. In particular, this yields a refined version of \Cref{thm:max_degree_ergraph}.

\begin{thm}\label{thm:G_n_m_bins_balls}
	Let $m=m(n)=\bigo{n}$ and $\degreesequence=\degreesequence(n)=\left(\degree{1}{G}, \ldots, \degree{n}{G}\right)$ be the degree sequence of $G=G(n,m)\ur \mathcal G(n,m)$. Moreover, let $\location=\location(n)\sim \binsandballs{n}{2m}$, $\loadvector=\loadvector(n)=\loadvector(\location)$ be the vector of loads of $\location$, and $\varepsilon>0$. Then
	\begin{enumerate}
	\item\label{thm:G_n_m_bins_balls1}
	the degree sequence $\degreesequence$ is contiguous with respect to $\loadvector$, i.e.	$\contiguous{\degreesequence}{\loadvector}$;
	\item\label{thm:G_n_m_bins_balls2}
	\whp\ $\rounddown{\concentration{n}{2m}-\varepsilon}\lessorequal \maxdegree{G}\lessorequal \rounddown{\concentration{n}{2m}+\varepsilon}$.
	\end{enumerate}

\end{thm}
\begin{proof}
We consider the random multigraph $\multigraph$ given by $\vertexSet{\multigraph}=[n]$ and $\edgeSet{\multigraph}=\setbuilder{\left\{\locationBit_{2i-1}, \locationBit_{2i}\right\}}{i \in [m]}$, where $\location=\left(\locationBit_1, \ldots, \locationBit_{2m}\right)$ is the location vector (see \Cref{fig:Gnm_balls_bins} for an illustration). We observe that for $v\in[n]$ the load $\load_v$ equals the degree $\degree{v}{\multigraph}$. For each graph $H\in \mathcal{G}(n,m)$ we have $\prob{\multigraph=H}=2^mm!/n^{2m}$. Hence, conditioned on the event that $\multigraph$ is simple, $\multigraph$ is distributed like $G$. Moreover, for $n$ large enough we have \begin{align*}
	\prob{\multigraph \text{ is simple}}&\equal\prob{\multigraph \text{ has no loops}}\cdot\condprob{\multigraph \text{ has no multiple edges}}{\multigraph \text{ has no loops}}
	\\
	&\equal
	\left(1-\frac{1}{n}\right)^m\hspace{0.05cm}\cdot\hspace{0.05cm} \prod_{i=0}^{m-1}\left(1-\frac{i}{\binom{n}{2}}\right)
	\greaterorequal
	\exp\left(\frac{-2m}{n}-\frac{4m^2}{n^2}\right)\greater\rho,
\end{align*}
for a suitable chosen $\rho>0$, since $m=\bigo{n}$. This shows $\liminf_{n\to \infty}\prob{\multigraph \text{ is simple}}>0$. Thus, each property that holds \whp\ in $\multigraph$, is also true \whp\ in $G$. In particular, the degree sequence $\degreesequence$ of $G$ is contiguous with respect to the degree sequence $\loadvector$ of $\multigraph$, i.e. $\contiguous{\degreesequence}{\loadvector}$. Together with \Cref{thm:concentration_balls_bins} this yields \whp\ $\rounddown{\concentration{n}{2m}-\varepsilon}\lessorequal \maxdegree{G}\lessorequal \rounddown{\concentration{n}{2m}+\varepsilon}$, as desired.
\end{proof}

We recall that we denote by $\nocomplex(n,m)$ a graph chosen uniformly at random from the class $\nocomplexclass(n,m)$ consisting of graphs having no complex components, vertex set $[n]$, and $m$ edges. Later $\nocomplex(n,m)$ will take the role of the non-complex part of the random planar graph. In this case the relation $m=n/2+\bigo{n^{2/3}}$ is satisfied (see \Cref{thm:internal_structure}). Britikov \cite{uni} provided a useful relation between $\nocomplex(n,m)$ and $G(n,m)$ in this range. 

\begin{thm}[\hspace{1sp}\cite{uni}]\label{thm:non_complex}
	Let $m=m(n)\leq n/2+\bigo{n^{2/3}}$ and $G=G(n,m)\ur \mathcal G(n,m)$. Then 
	\begin{align*}
	\liminf_{n \to \infty}~ \prob{G \text{ has no complex components }}>0.
	\end{align*}
In particular, $\liminf_{n \to \infty}~ \prob{G \text{ is planar }}>0$.
\end{thm}

Combining Theorems \ref{thm:G_n_m_bins_balls}\ref{thm:G_n_m_bins_balls2} and \ref{thm:non_complex} we can deduce that \whp\ $\maxdegree{\nocomplex(n,m)}$ is concentrated at two values.

\begin{lem}\label{lem:random_non_complex}
	Let $m=m(n)= n/2+\bigo{n^{2/3}}$, $\nocomplex=\nocomplex(n,m)\ur \nocomplexclass(n,m)$ be a random graph without complex components, and $\varepsilon>0$. Then \whp\ $\rounddown{\specialconcentration{n}-\varepsilon}\lessorequal \maxdegree{\nocomplex}\lessorequal \rounddown{\specialconcentration{n}+\varepsilon}$.
\end{lem}
\begin{proof}
Combining Theorems \ref{thm:G_n_m_bins_balls}\ref{thm:G_n_m_bins_balls2} and \ref{thm:non_complex} yields \whp
\begin{align}\label{eq:17}
	\rounddown{\concentration{n}{2m}-\varepsilon/2}\lessorequal \maxdegree{\nocomplex}\lessorequal \rounddown{\concentration{n}{2m}+\varepsilon/2}.
\end{align}
Using \Cref{lem:nu}\ref{lem:nu3} we obtain $\concentration{n}{2m}=\specialconcentration{n}+\smallo{1}$. Together with (\ref{eq:17}) this shows the statement. 
\end{proof}

\section{Random complex part and forests with specified roots}\label{sec:forests}
The goal of this section is to prove that \whp\ the maximum degree of a random complex part is concentrated at two values (see \Cref{thm:random_complex_part}\ref{thm:random_complex_part2}). As a random complex part can be constructed by using a random forest, we start by analysing the class $\forestclass(n,\ntrees)$ of forests on vertex set $[n]$ having $\ntrees$ tree components (some of which might just be isolated vertices) such that the vertices $1, \ldots, \ntrees$ lie all in different tree components. 

In \Cref{sub:pruefer} we generalise the concept of Prüfer sequences to forests. Then we determine the maximum degree in a random forest in \Cref{sub:random_forest}. Finally, we derive the concentration result on the maximum degree in a random complex part in \Cref{sub:random_complex_part}.

\subsection{Prüfer sequences for forests with specified roots}\label{sub:pruefer}
Similar to Prüfer sequences for trees (see e.g. \cite{book_pruefer, book_pruefer2}), there is a bijection between $\forestclass(n, \ntrees)$ and $\sequences{n}{\ntrees}:=[n]^{n-\ntrees-1}\times [\ntrees]$ (see e.g. \cite[Section 6.6]{Asymptopia}): Given a forest $\forest \in \forestclass(n,\ntrees)$ we construct a sequence $\left(\forest_0, \ldots, \forest_{n-\ntrees}\right)$ of forests and two sequences $\left(x_1, \ldots, x_{n-\ntrees}\right)$ and $\left(y_1, \ldots, y_{n-\ntrees}\right)$ of vertices as follows. We start with $\forest_0:=\forest$. Given $\forest_{i-1}$ for an $i\in[n-\ntrees]$, we let $y_i$ be the leaf with largest label in $\forest_{i-1}$ and $x_i$ be the unique neighbour of $y_i$. Furthermore, we obtain $\forest_i$ by deleting the edge $x_iy_i$ in $\forest_{i-1}$. We note that this construction is always possible, since $\forest_{i-1}$ has $n-t-i+1$ edges and therefore at least one leaf. We call 
\begin{align}\label{eq:20}
\prueferseqence(\forest):=\left(x_1, \ldots, x_{n-\ntrees}\right)
\end{align}
the {\em Prüfer sequence} of $\forest$. We denote by $\frequency{v}{\mathbf{w}}:=\left|\setbuilder{i\in[n-\ntrees]}{w_i=v}\right|$ the number of occurrences of an element $v\in[n]$ in $\mathbf{w}=\left(w_1, \ldots, w_{n-\ntrees}\right)\in[n]^{n-\ntrees}$. 

\begin{thm}\label{thm:pruefer}
Let $n, \ntrees\in \N$ and $\forestclass(n,\ntrees)$ be the class of forests on vertex set $[n]$ consisting of $\ntrees$ tree components such that the vertices $1, \ldots, \ntrees$ lie all in different tree components. In addition, let $\sequences{n}{\ntrees}=[n]^{n-\ntrees-1}\times [\ntrees]$ and $\prueferseqence(\forest)$ be the Prüfer sequence of $F\in\forestclass(n,\ntrees)$ as defined in (\ref{eq:20}). Then $\prueferseqence:\forestclass(n,\ntrees)\to \sequences{n}{\ntrees}$ is a bijection. For $\forest\in \forestclass(n,\ntrees)$ and $v\in[n]$ we have \begin{align}\label{eq:12}
	\degree{v}{\forest}\equal
	\begin{cases}
		\frequency{v}{\prueferseqence(\forest)} & \text{~~if~~} v\in [\ntrees]
		\\
		\frequency{v}{\prueferseqence(\forest)}+1 & \text{~~if~~} v\in[n]\setminus [\ntrees].
	\end{cases}
\end{align}
\end{thm}
\Cref{thm:pruefer} can be shown by using similar ideas as in the classical case of trees. For the sake of completeness, we provide a proof of \Cref{thm:pruefer} in \Cref{sec:proof_pruefer}.

\subsection{Degree sequence and maximum degree of a random forest}\label{sub:random_forest}
We consider a random forest $\forest=\forest(n,\ntrees)\ur \forestclass(n,\ntrees)$ and determine the degree sequence of $\forest$ and the maximum degree $\maxdegree{\forest}$. 

\begin{thm}\label{thm:forest_balls_bins}
Let $n, \ntrees\in \N$ and $\degreesequence=\left(\degree{1}{\forest}, \ldots, \degree{n}{\forest}\right)$ be the degree sequence of $\forest=\forest\left(n, \ntrees\right)\ur \forestclass\left(n, \ntrees\right)$. Let $\location\sim \binsandballs{n}{n-\ntrees-1}$ and $\load_j=\load_j(\location)$ be the load in bin $\bin_j$ for each $j\in[n]$. In addition, let $\rootrv\ur [\ntrees]$ (independent of $\forest$) and for $j\in[\ntrees]$ we define $\rootvectorbit_j=1$ if $\rootrv=j$ and $\rootvectorbit_j=0$ otherwise. Then
\begin{align*}
\big(\load_1+\rootvectorbit_1, \ldots, \load_\ntrees+\rootvectorbit_\ntrees, \load_{\ntrees+1}+1, \ldots, \load_n+1\big)\sim \degreesequence.
\end{align*}
\end{thm}

\begin{proof}
Instead of directly choosing $\forest$ from $\forestclass(n,\ntrees)$, we can equivalently create $\forest$ by Prüfer sequences from \Cref{sub:pruefer}: First we perform a balls-into-bins experiment with $n$ bins and $n-\ntrees-1$ balls and let $\location=\left(\locationBit_1, \ldots, \locationBit_{n-\ntrees-1}\right)\sim\binsandballs{n}{n-\ntrees-1}$ be the location vector. Then we independently choose $\locationBit_{n-\ntrees}\ur [\ntrees]$ and set $\forest=\prueferinvers\left(\locationBit_1, \ldots, \locationBit_{n-\ntrees}\right)$ and the statement follows by (\ref{eq:12}).
\end{proof}

Using this connection to the balls-into-bins model we obtain an upper bound on $\maxdegree{\forest(n,\ntrees)}$ (see \Cref{thm:forest_max_degree}\ref{thm:forest_max_degree1}). If we assume that $\ntrees$ is not too \lq large\rq, we can even show that \whp\ $\maxdegree{\forest(n,\ntrees)}$ is concentrated at two values and that the maximum degree of a root vertex, i.e. a vertex in $[\ntrees]$, is much smaller than $\maxdegree{\forest(n,\ntrees)}$ (see \Cref{thm:forest_max_degree}\ref{thm:forest_max_degree2}). We will need these facts later when we use random forests to build a random complex part (see \Cref{sub:random_complex_part}).
\begin{thm}\label{thm:forest_max_degree}
Let $\ntrees=\ntrees(n)$, $\forest=\forest\left(n, \ntrees\right)\ur \forestclass\left(n, \ntrees\right)$, and $\varepsilon>0$. Then 
\begin{enumerate}
\item\label{thm:forest_max_degree1}
\whp\
$\maxdegree{\forest}\lessorequal \rounddown{\specialconcentration{n}}+2$;
\item\label{thm:forest_max_degree2}
if $\ntrees=\bigo{n^{1-\beta}}$ for some $\beta>0$, then \whp\ $\rounddown{\specialconcentration{\nbins}-\varepsilon}+1\lessorequal \maxdegree{\forest}\lessorequal \rounddown{\specialconcentration{\nbins}+\varepsilon}+1$ and $\maxdegree{\forest}-\max\setbuilder{d_\forest(r)}{r \in [\ntrees]}=\smallomega{1}$.
\end{enumerate}
\end{thm}

\begin{proof}
Let $\location\sim \binsandballs{n}{n-\ntrees-1}$, $\maxload=\maxload(\location)\sim \maxbinsandballs{n}{n-\ntrees-1}$ be the maximum load of $\location$, and $\maxload_{t}=\maxload_{t}(\location)$ be the maximum load of one of the first $t$ bins of $\location$. Due to \Cref{thm:concentration_balls_bins} we have \whp\ $\maxload\leq \rounddown{\concentration{n}{n-\ntrees-1}}+1\leq\rounddown{\specialconcentration{n}}+1$, where we used in the last inequality \Cref{lem:nu}\ref{lem:nu4}. Combining it with \Cref{thm:forest_balls_bins} we have 
\begin{align*}
\probLarge{\maxdegree{\forest}> \rounddown{\specialconcentration{n}}+2}\lessorequal \prob{\maxload >\rounddown{\specialconcentration{n}}+1}=\smallo{1}.
\end{align*}
This shows statement \ref{thm:forest_max_degree1}.

By \Cref{lem:max_load_subset} we have \whp\ $\maxload-\maxload_{t}=\smallomega{1}$. This together with \Cref{thm:forest_balls_bins} implies \whp\ $\maxdegree{\forest}-\max\setbuilder{d_\forest(r)}{r \in [\ntrees]}=\smallomega{1}$ and $\contiguous{\maxdegree{\forest}}{\maxload+1}$. Thus, we obtain by \Cref{thm:concentration_balls_bins} that \whp
\begin{align*}
	\rounddown{\concentration{n}{n-\ntrees-1}-\varepsilon/2}+1\lessorequal \maxdegree{\forest}\lessorequal \rounddown{\concentration{n}{n-\ntrees-1}+\varepsilon/2}+1.
\end{align*}
By \Cref{lem:nu}\ref{lem:nu3} we have $\concentration{n}{n-\ntrees-1}=\specialconcentration{n}+\smallo{1}$, which proves statement \ref{thm:forest_max_degree2}.
\end{proof}

We note that the special case of random trees, i.e. when $\ntrees=1$, was studied in \cite{moon,CGS}. In particular, Carr, Goh, and Schmutz \cite{CGS} used the saddle-point method to show that \whp\ the maximum degree in random trees is concentrated at two values.

\subsection{Random complex part}\label{sub:random_complex_part}
We consider the class $\complexclass\left(C,q\right)$ consisting of complex graphs with core $C$ and vertex set $[q]$, where $C$ is a given core and $q\in \N$ (cf. \Cref{def:random_complex_part}). As illustrated in \Cref{fig:construction_complex_part}, we can construct $\complexgraph\left(C,q\right)\ur \complexclass\left(C,q\right)$ via the balls-into-bins model. Assuming that $\maxdegree{C}$ is bounded and $\numberVertices{C}$ is \lq small\rq\ compared to $q$, we will use \Cref{thm:forest_max_degree} to show that the maximum degree of $\complexgraph\left(C,q\right)$ is strongly concentrated.

\setlength\tabcolsep{0.2cm}
\begin{table}
	\begin{tabular}[t]{c|c}
		\begin{tikzpicture}[very thick, every node/.style={rectangle, inner sep=0cm, minimum size=0.4cm,fill=lightgray}]
			\small
			\node (1) at (0,0) [draw] {\textbf{1}};
			\node (2) at (1.8,0) [draw] {\textbf{2}};
			\node (3) at (0.9,1.75) [draw] {\textbf{3}};
			\node (4) at (2.6,0) [draw] {\textbf{4}};
			\node (5) at (3.2,0) [draw] {\textbf{5}};
			\node (6) at (2.6,0.7) [draw] {\textbf{6}};	
			\node (7) at (3.2,0.7) [draw] {\textbf{7}};
			\node (8) at (2.6,1.4) [draw] {\textbf{8}};
			\node (9) at (3.2,1.4) [draw] {\textbf{9}};	
			
			\draw[-] (1) to (2);
			\draw[-] (1) to (3);
			\draw[-] (2) to (3);
			
			\node () at (1.55,-0.85) [draw=none, align=left,fill=white] {(a) {\em Given}:\\Core $C=\text{\lq triangle\rq}$, $q=9$};
			
		\end{tikzpicture} &
		\begin{tikzpicture}[very thick, every node/.style={rectangle, inner sep=0cm, minimum size=0.4cm,fill=lightgray}]
			\small			
			\draw[-] (-0.02,-0.1) to (0.92,-0.1);
			\draw[-] (0,-0.1) to (0,1.4);
			\draw[-] (0.9,-0.1) to (0.9,1.4);
			\node (1) at (0.45,-0.5) [draw] {\textbf{1}};
			
			\draw[-] (1.08,-0.1) to (2.02,-0.1);
			\draw[-] (1.1,-0.1) to (1.1,1.4);
			\draw[-] (2,-0.1) to (2,1.4);
			\node (2) at (1.55,-0.5) [draw] {\textbf{2}};
			
			\draw[-] (2.18,-0.1) to (3.12,-0.1);
			\draw[-] (2.2,-0.1) to (2.2,1.4);
			\draw[-] (3.1,-0.1) to (3.1,1.4);
			\node (3) at (2.65,-0.5) [draw] {\textbf{3}};
			
			\draw[-] (3.28,-0.1) to (4.22,-0.1);
			\draw[-] (3.3,-0.1) to (3.3,1.4);
			\draw[-] (4.2,-0.1) to (4.2,1.4);
			\node (4) at (3.75,-0.5) [draw] {\textbf{4}};
			
			\draw[-] (4.38,-0.1) to (5.32,-0.1);
			\draw[-] (4.4,-0.1) to (4.4,1.4);
			\draw[-] (5.3,-0.1) to (5.3,1.4);
			\node (5) at (4.85,-0.5) [draw] {\textbf{5}};
			
			\draw[-] (5.48,-0.1) to (6.42,-0.1);
			\draw[-] (5.5,-0.1) to (5.5,1.4);
			\draw[-] (6.4,-0.1) to (6.4,1.4);
			\node (6) at (5.95,-0.5) [draw] {\textbf{6}};
			
			\draw[-] (6.58,-0.1) to (7.52,-0.1);
			\draw[-] (6.6,-0.1) to (6.6,1.4);
			\draw[-] (7.5,-0.1) to (7.5,1.4);
			\node (7) at (7.05,-0.5) [draw] {\textbf{7}};
			
			\draw[-] (7.68,-0.1) to (8.62,-0.1);
			\draw[-] (7.7,-0.1) to (7.7,1.4);
			\draw[-] (8.6,-0.1) to (8.6,1.4);
			\node (8) at (8.15,-0.5) [draw] {\textbf{8}};
			
			\draw[-] (8.78,-0.1) to (9.72,-0.1);
			\draw[-] (8.8,-0.1) to (8.8,1.4);
			\draw[-] (9.7,-0.1) to (9.7,1.4);
			\node (9) at (9.25,-0.5) [draw] {\textbf{9}};
			
			\node (B1) at (7.94,0.15) [circle, draw, minimum width=0cm, minimum height=0cm, inner sep=0.02cm] {\textbf{3}};	
			\node (B2) at (9.25,0.15) [circle, draw, minimum width=0cm, minimum height=0cm, inner sep=0.02cm] {\textbf{2}};	
			\node (B3) at (8.36,0.15) [circle, draw, minimum width=0cm, minimum height=0cm, inner sep=0.02cm] {\textbf{5}};
			\node (B4) at (0.45,0.15) [circle, draw, minimum width=0cm, minimum height=0cm, inner sep=0.02cm] {\textbf{4}};
			\node (B5) at (3.75,0.15) [circle, draw, minimum width=0cm, minimum height=0cm, inner sep=0.02cm] {\textbf{1}};	
			\node (B6) at (1.55,0.15) [circle, draw, minimum width=0cm, minimum height=0cm, inner sep=0.03cm, dashed] {\textbf{6}};		
			
			\node () at (4.9,-1.4) [draw=none, align=left,fill=white] {(b) {\em Balls into bins}: $q=9$ bins, $q-\numberVertices{C}-1=5$ balls, and a ball (6)\\ that can be allocated only to one of the first $\numberVertices{C}=3$ bins};
			
		\end{tikzpicture}
		\\ \hline
		\multicolumn{2}{l}{	
			\begin{tabular}[t]{c|c|c}
				\begin{tikzpicture}[very thick, every node/.style={rectangle, inner sep=0cm, minimum size=0.4cm,fill=lightgray}]
					\small
					\node () at (-1.1,0.8) [draw, align=left] {\textbf{4}};
					\node () at (-0.55,0.8) [draw, align=left] {\textbf{9}};
					\node () at (0,0.8) [draw, align=left] {\textbf{8}};
					\node () at (0.55,0.8) [draw, align=left] {\textbf{1}};
					\node () at (1.1,0.8) [draw, align=left] {\textbf{8}};
					\node () at (1.65,0.8) [draw, align=left] {\textbf{2}};
					\node () at (0,2.3) [draw=none, align=left,fill=white] {};	
					\node () at (0.2,-0.7) [draw=none, align=left,fill=white] {(c) {\em Prüfer sequence}};
				\end{tikzpicture}
				&
				\begin{tikzpicture}[very thick, every node/.style={rectangle, inner sep=0cm, minimum size=0.4cm,fill=lightgray}]
					\small
					\node (1) at (0,2) [draw] {\textbf{1}};
					\node (2) at (0,1) [draw] {\textbf{2}};
					\node (3) at (0,0) [draw] {\textbf{3}};
					\node (4) at (2,0.4) [draw] {\textbf{4}};
					\node (5) at (1,2) [draw] {\textbf{5}};
					\node (6) at (3,1.6) [draw] {\textbf{6}};
					\node (7) at (3,0.4) [draw] {\textbf{7}};
					\node (8) at (1,1) [draw] {\textbf{8}};	
					\node (9) at (2,1.6) [draw] {\textbf{9}};	
					
					\draw[-] (1) to (5);
					\draw[-] (2) to (8);
					\draw[-] (8) to (4);
					\draw[-] (8) to (9);
					\draw[-] (4) to (7);
					\draw[-] (6) to (9);
					\node () at (1.5,-0.7) [draw=none, align=left,fill=white] {(d) {\em Random forest} $F(9,3)$};
					
				\end{tikzpicture} &
				\begin{tikzpicture}[very thick, every node/.style={rectangle, inner sep=0cm, minimum size=0.4cm,fill=lightgray}]
					\small
					\node (1) at (0,0) [draw] {\textbf{1}};
					\node (2) at (1.8,0) [draw] {\textbf{2}};
					\node (3) at (0.9,1.85) [draw] {\textbf{3}};
					\node (4) at (3.8,0) [draw] {\textbf{4}};
					\node (5) at (-0.3,1.1) [draw] {\textbf{5}};
					\node (6) at (4.8,1.2) [draw] {\textbf{6}};
					\node (7) at (4.8,0) [draw] {\textbf{7}};
					\node (8) at (2.8,0.6) [draw] {\textbf{8}};	
					\node (9) at (3.8,1.2) [draw] {\textbf{9}};	
					
					\draw[-] (1) to (2);
					\draw[-] (1) to (3);
					\draw[-] (2) to (3);
					\draw[-] (1) to (5);
					\draw[-] (2) to (8);
					\draw[-] (8) to (4);
					\draw[-] (8) to (9);
					\draw[-] (4) to (7);
					\draw[-] (6) to (9);	
					
					\node () at (2.4,-0.7) [draw=none, align=left,fill=white] {(e) {\em Random complex part} $Q(C,q)$};
					
				\end{tikzpicture}
			\end{tabular}
		}
	\end{tabular}
	
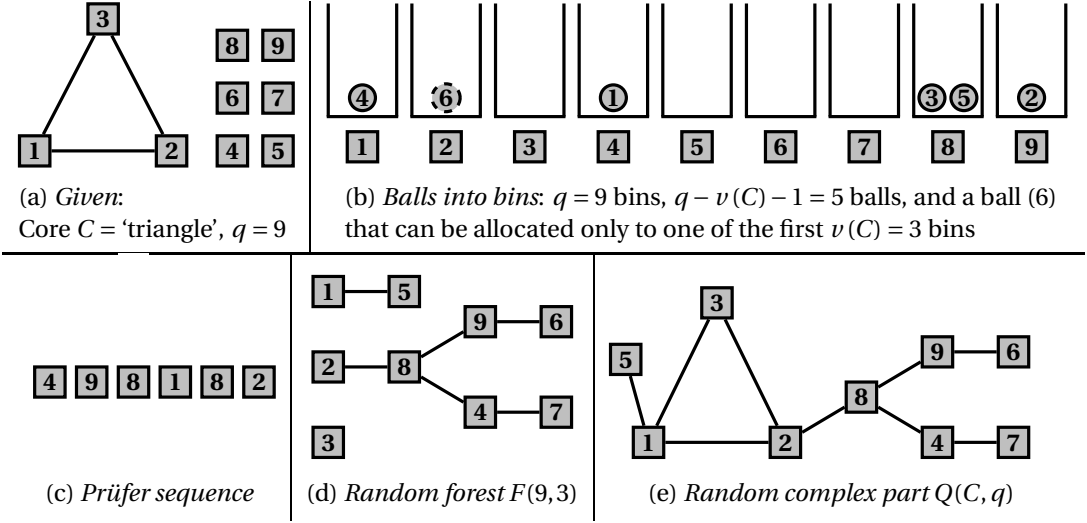
\captionof{figure}{Construction of the random complex part $\complexgraph(C,q)$: (a) Given a core $C$ and $q\in\N$, (b) a balls-into-bins experiment is translated (deterministically) to (c) a Prüfer sequence, (d) a random forest, and finally to (e) the random complex part $\complexgraph(C,q)$.}\label{fig:construction_complex_part}
\end{table}

\begin{thm}\label{thm:random_complex_part}
For each $n\in \N$, let $C=C(n)$ be a core and $q=q(n)\in \N$. In addition, let $\complexgraph=\complexgraph\left(C,q\right)\ur \complexclass\left(C,q\right)$ be a random complex part with core $C$ and vertex set $[q]$ as in \Cref{def:random_complex_part} and $\varepsilon>0$. If $\maxdegree{C}=\Th{1}$, then the following hold.
\begin{enumerate}
\item\label{thm:random_complex_part1}
\Whp\ $\maxdegree{Q}\lessorequal \specialconcentration{q}+\bigo{1}$.
\item\label{thm:random_complex_part2}
If in addition $\numberVertices{C}=\bigo{q^{1-\beta}}$ for some $\beta>0$, then \whp\ $\rounddown{\specialconcentration{q}-\varepsilon}+1\lessorequal \maxdegree{Q}\lessorequal \rounddown{\specialconcentration{q}+\varepsilon}+1$.
\end{enumerate}
\end{thm}

\begin{proof}
We observe that $Q$ can be obtained by choosing a random forest $\forest=\forest(q,\numberVertices{C})\ur\forestclass(q,\numberVertices{C})$ and then replacing each vertex $r$ in $C$ by the tree component of $\forest$ with root $r$. For a vertex $v\in \vertexSet{Q}$ we have
\begin{align}\label{eq:3}
	\degree{v}{Q}=
	\begin{cases}
		\degree{v}{C}+\degree{v}{F}& \text{if}~~ v \in \vertexSet{C} \\
		\degree{v}{F} & \text{ otherwise}.
	\end{cases}
\end{align}
Hence, we have $\maxdegree{Q}\lessorequal \maxdegree{C}+\maxdegree{\forest}$. By \Cref{thm:forest_max_degree}\ref{thm:forest_max_degree1} we get \whp\ $\maxdegree{\forest}\lessorequal \specialconcentration{q}+2$. Together with the fact $\maxdegree{C}=\Th{1}$ this yields statement \ref{thm:random_complex_part1}. For \ref{thm:random_complex_part2} we apply \Cref{thm:forest_max_degree}\ref{thm:forest_max_degree2} to $\forest$. Together with (\ref{eq:3}) and $\maxdegree{C}=\Th{1}$ this implies \whp\ $\maxdegree{Q}=\maxdegree{\forest}$. Thus, statement \ref{thm:random_complex_part2} follows by applying again \Cref{thm:forest_max_degree}\ref{thm:forest_max_degree2}.
\end{proof}

\section{Proofs of \Cref{thm:main_sub,thm:main} and \Cref{cor:maxdegree,cor:independent}}\label{sec:proof}
Throughout this section, let $\planargraph=\planargraph(n,m)\ur \planarclass(n,m)$ be the random planar graph.
\proofof{thm:main_sub}
In \Cref{thm:non_complex} we have seen that $\liminf_{n \to \infty} \prob{G(n,m)\text{ is planar}}>0$. Thus, each graph property that holds \whp\ in $G(n,m)$ is also true \whp\ in $\planargraph$ and the first statement follows by \Cref{thm:G_n_m_bins_balls}\ref{thm:G_n_m_bins_balls2}. By taking $\varepsilon=1/3$ we get the \lq in particular\rq\ statement. \qed

\proofof{thm:main}
We split $\planargraph$ into the large complex part $\Complexlargestcore=\complexlargestcore{\planargraph}$, the small complex part $\Complexrestcore=\complexrestcore{\planargraph}$, and the non-complex part $\Restcomplex=\restcomplex{\planargraph}$ as described in \Cref{sub:decomposition}. We claim that \whp\ the following hold.
\begin{enumerate}[label=(\roman*)]
\item\label{cl:1}
$\rounddown{\specialconcentration{\funcL}-\varepsilon}+1\lessorequal\maxdegree{\Complexlargestcore}\lessorequal\rounddown{\specialconcentration{\funcL}+\varepsilon}+1$;
\item\label{cl:2}
$\maxdegree{\Complexrestcore}\lessorequal\specialconcentration{\funcR^{2/3}}+\bigo{1}$;
\item\label{cl:3}
$\rounddown{\specialconcentration{\funcR}-\varepsilon}\lessorequal\maxdegree{\Restcomplex}\lessorequal\rounddown{\specialconcentration{\funcR}+\varepsilon}$.
\end{enumerate}
Assuming these three claims are true we can finish the proof as follows. By \Cref{thm:internal_structure}\ref{thm:internal_structure3} we have \whp\ $\Largestcomponent=\Complexlargestcore$ and therefore also \whp\ $\Rest=\Complexrestcore\cup \Restcomplex$. Thus, statement \ref{thm:main1} of \Cref{thm:main} follows by \ref{cl:1}. By \Cref{lem:nu}\ref{lem:nu1} we have $\specialconcentration{\funcR^{2/3}}=\left(2/3+\smallo{1}\right)\log \funcR/\log\log \funcR$ and $\specialconcentration{\funcR}=\left(1+\smallo{1}\right)\log \funcR/\log\log \funcR$. Combining that with \ref{cl:2} and \ref{cl:3} yields \whp\ $\maxdegree{\Complexrestcore\cup \Restcomplex}=\maxdegree{\Restcomplex}$ and therefore also \whp\ $\maxdegree{\Rest}=\maxdegree{\Restcomplex}$. Hence, statement \ref{thm:main2} of \Cref{thm:main} follows by \ref{cl:3}. Finally, we obtain the \lq in particular\rq\ statement by taking $\varepsilon=1/3$.

To prove the claims \ref{cl:1}--\ref{cl:3}, we will follow the strategy described in \Cref{sec:strategy}: We will construct a conditional random graph $\condGraph{\randomGraph}{\seq}$ which is distributed like the random graph $\tilde{\planargraph}$ introduced in \Cref{sub:strategy_decomposition}. Then we will determine the maximum degrees of the large complex part, small complex part and non-complex part of $\condGraph{\randomGraph}{\seq}$ (or equivalently of $\tilde{\planargraph}$). Finally, we will apply \Cref{lem:conditional_random_graphs} to translate these results to the random planar graph $\planargraph$. 

Let $\beta:=\min\left\{\delta/2,1/5\right\}$ and $\cl(n)$ be the subclass of $\planarclass(n,m)$ consisting of those graphs $H$ satisfying
\begin{align}
\maxdegreeLarge{\core{H}}&\equal \Th{1}, \label{eq:7}\\
\numberVerticesLarge{\largestcomponent{\core{H}}}&\equal\bigo{\funcL^{1-\beta}},\label{eq:8}\\
\numberVerticesLarge{\complexlargestcore{H}}&\equal\Th{\funcL},\label{eq:9}\\
\numberVerticesLarge{\complexrestcore{H}}&\equal\bigo{\funcR^{2/3}},\label{eq:10}
\\
\numberVerticesLarge{\restcomplex{H}}&\equal\Th{\funcR},\label{eq:11}
\\
\numberEdgesLarge{\restcomplex{H}}&\equal\numberVerticesLarge{\restcomplex{H}}/2+\bigo{\numberVerticesLarge{\restcomplex{H}}^{2/3}}.\label{eq:11b}
\end{align}
Due to \Cref{thm:internal_structure} we can choose the implicit hidden constants in the equations (\ref{eq:7})--(\ref{eq:11b}) such that $\planargraph\in\cl(n)$ with a probability of at least $1-\gamma/2$, for arbitrary $\gamma>0$. We will apply \Cref{lem:conditional_random_graphs} to the class $\cl:=\bigcup_{n\in\N}\cl(n)$. To that end, we define the function $\func$ such that for $H\in\cl$ we have
\begin{align*}
\func(H):=\big(\core{H}, \numberVertices{\complexlargestcore{H}}, \numberVertices{\complexrestcore{H}}\big).
\end{align*}
Let $\seq=\left(C_n, \lambda_n, \sigma_n\right)_{n\in\N}$ be a sequence that is feasible for $(\cl, \func)$ and let $\randomGraph=\randomGraph(n)\ur \cl(n)$. 
By definition the possible realisations of $\condGraph{\randomGraph}{\seq}$ are those graphs $H \in \cl$ with $\core{H}=C_n$, $\numberVertices{\complexlargestcore{H}}=\lambda_n$, and $\numberVertices{\complexrestcore{H}}=\sigma_n$. Hence, $\condGraph{\randomGraph}{\seq}=\complexlargestcore{\condGraph{\randomGraph}{\seq}}~\dot\cup~ \complexrestcore{\condGraph{\randomGraph}{\seq}}~\dot\cup~ \restcomplex{\condGraph{\randomGraph}{\seq}}$ can be constructed as follows. For $\complexlargestcore{\condGraph{\randomGraph}{\seq}}$ we choose uniformly at random a complex graph with $\lambda_n$ vertices and core $\largestcomponent{C_n}$ and for $\complexrestcore{\condGraph{\randomGraph}{\seq}}$ a complex graph with $\sigma_n$ vertices and core $\rest{C_n}$. For $\restcomplex{\condGraph{\randomGraph}{\seq}}$ we choose a graph without complex components having $u_n:=n-\lambda_n-\sigma_n$ vertices and $w_n:=m-\numberEdges{C_n}+\numberVertices{C_n}-\lambda_n-\sigma_n$ edges. Summing up, we have
\begin{align}
\complexlargestcoreLarge{\condGraph{\randomGraph}{\seq}} &\sim \complexgraph\big(\largestcomponent{C_n}, \lambda_n\big);\label{eq:14}\\
\complexrestcoreLarge{\condGraph{\randomGraph}{\seq}} &\sim \complexgraph\big(\rest{C_n}, \sigma_n\big);\label{eq:15}\\
\restcomplexLarge{\condGraph{\randomGraph}{\seq}}&\sim \nocomplex\big(u_n,w_n\big),\label{eq:16}
\end{align}
where the random complex parts and the random graph without complex components on the right hand side are as defined in \Cref{def:random_complex_part,def:nocomplex}, respectively.
Due to (\ref{eq:7}) and (\ref{eq:8}) we have $\maxdegree{C_n}=\Th{1}$ and $\numberVerticesLarge{\largestcomponent{C_n}}=\bigo{\lambda_n^{1-\beta}}$. Hence, we can apply \Cref{thm:random_complex_part}\ref{thm:random_complex_part2} to $\complexgraph\big(\largestcomponent{C_n}, \lambda_n\big)$. Together with (\ref{eq:14}) this implies \whp
\begin{align}\label{eq:4}
	\rounddown{\specialconcentration{\lambda_n}-\varepsilon/2}+1\lessorequal \maxdegreeLarge{\complexlargestcore{\condGraph{\randomGraph}{\seq}}}\lessorequal 	\rounddown{\specialconcentration{\lambda_n}+\varepsilon/2}+1.
\end{align}
Using \Cref{lem:nu}\ref{lem:nu2} we have $\specialconcentration{\lambda_n}=\specialconcentration{\funcL}+\smallo{1}$, as $\lambda_n=\Th{\funcL}$ by (\ref{eq:9}). Together with (\ref{eq:4}) this shows \whp
\begin{align}\label{eq:5}
	\rounddown{\specialconcentration{\funcL}-\varepsilon}+1\lessorequal \maxdegreeLarge{\complexlargestcore{\condGraph{\randomGraph}{\seq}}}\lessorequal 	\rounddown{\specialconcentration{\funcL}+\varepsilon}+1.
\end{align}
By \Cref{lem:conditional_random_graphs} inequality (\ref{eq:5}) is also \whp\ true if we replace $\condGraph{\randomGraph}{\seq}$ by $\randomGraph$. Combining it with the fact that $\planargraph \in \cl$ with probability at least $1-\gamma/2$ we obtain that with probability at least $1-\gamma$
\begin{align*}
	\rounddown{\specialconcentration{\funcL}-\varepsilon}+1\lessorequal \maxdegree{\Complexlargestcore}\lessorequal 	\rounddown{\specialconcentration{\funcL}+\varepsilon}+1
\end{align*}
for all $n$ large enough. As $\gamma>0$ was arbitrary, this shows claim \ref{cl:1}. 

Next, we prove claims \ref{cl:2} and \ref{cl:3} in a similar fashion. Combining (\ref{eq:15}) and \Cref{thm:random_complex_part}\ref{thm:random_complex_part1} yields $\maxdegreeLarge{\complexrestcore{\condGraph{\randomGraph}{\seq}}}\leq \specialconcentration{\sigma_n}+\bigo{1}$. Due to \Cref{lem:nu}\ref{lem:nu9} and \ref{lem:nu2} we have $\specialconcentration{\sigma_n}\leq\specialconcentration{\funcR^{2/3}}+o(1)$, where we used $\sigma_n=\bigo{\funcR^{2/3}}$ by (\ref{eq:10}). This yields $\maxdegreeLarge{\complexrestcore{\condGraph{\randomGraph}{\seq}}}\leq\specialconcentration{\funcR^{2/3}}+\bigo{1}$. Thus, claim \ref{cl:2} follows by \Cref{lem:conditional_random_graphs}. 
Due to (\ref{eq:11b}) we have $w_n=u_n/2+\bigo{u_n^{2/3}}$. Hence, we can combine (\ref{eq:16}) and \Cref{lem:random_non_complex} to obtain \whp
\begin{align}\label{eq:6}
\rounddown{\specialconcentration{u_n}-\varepsilon/2}\lessorequal \maxdegreeLarge{\restcomplex{\condGraph{\randomGraph}{\seq}}}\lessorequal \rounddown{\specialconcentration{u_n}+\varepsilon/2}.
\end{align}
Due to (\ref{eq:11}) we have $u_n=\Th{\funcR}$ and therefore, we obtain $\specialconcentration{u_n}=\specialconcentration{\funcR}+\smallo{1}$ by \Cref{lem:nu}\ref{lem:nu2}. Using that in (\ref{eq:6}) we get \whp
 \begin{align*}
 \rounddown{\specialconcentration{\funcR}-\varepsilon}\lessorequal \maxdegreeLarge{\restcomplex{\condGraph{\randomGraph}{\seq}}}\lessorequal \rounddown{\specialconcentration{\funcR}+\varepsilon}.
 \end{align*}
 Now claim \ref{cl:3} follows by \Cref{lem:conditional_random_graphs}. \qed

\proofof{cor:maxdegree}
We distinguish two cases. If $m$ is as in \Cref{thm:main_sub}, then \whp\ $\maxdegree{\planargraph}=\concentration{n}{2m}+\bigo{1}=\left(1+\smallo{1}\right)\log n/\log\log n$, where we used \Cref{thm:main_sub}, \Cref{lem:nu}\ref{lem:nu1}, and that $m=\Th{n}$. If $m$ is as in \Cref{thm:main}, then $\max\left\{\funcL(n),\funcR(n)\right\}=n$. Together with \Cref{lem:nu}\ref{lem:nu1} and \ref{lem:nu9} this implies $\max\left\{\specialconcentration{\funcL},\specialconcentration{\funcR}\right\}=\specialconcentration{n}=\left(1+\smallo{1}\right)\log n/\log\log n$. Combining that with the \lq in particular\rq\ statement of \Cref{thm:main} we obtain \whp\ $\maxdegree{\planargraph}=\left(1+\smallo{1}\right)\log n/\log\log n$, as desired.
\qed

\proofof{cor:independent}
The assertion follows directly from \Cref{thm:main} and the fact from \Cref{def:functionLR}(b) that $\funcL=n$ and $\funcR=n$. \qed

\section{Non-concentration of the maximum degree: Proof of \Cref{thm:main_dense}}\label{sec:planar_dense}
We recall that we denote by $\planarclass(n,m)$ the class of all vertex-labelled planar graphs on vertex set $[n]$ with $m$ edges. Furthermore, let $\planarclass_C(n,m)\subseteq\planarclass(n,m)$ be the subclass containing all {\em connected} planar graphs on vertex set $[n]$ with $m$ edges. Our starting point is the following result of Gim\'enez and Noy \cite{gim}.
\begin{thm}[\hspace{1sp}\cite{gim}]\label{thm:number_planar}
Let $\ratio \in (1,3)$ and $m=m(n)=\rounddown{\ratio n}$. Then there exist constants $\gamma, u>0$ and $c\geq c_1>0$ such that
\begin{align*}
\left|\planarclass(n,m)\right|&=\left(1+\smallo{1}\right)c n^{-4}\gamma^n u^m n!;\\
\left|\planarclass_C(n,m)\right|&=\left(1+\smallo{1}\right)c_1 n^{-4}\gamma^n u^m n!.
\end{align*}
\end{thm}

Using \Cref{thm:number_planar} we can show that the probability that the dense random planar graph $\planargraph(n,m)$ has $\niv$ isolated vertices and $\nie$ isolated edges is bounded away from 0, for each fixed $\niv,\nie\in\N_0:=\N\cup\left\{0\right\}$.
\begin{lem}\label{lem:isolated}
Let $\planargraph=\planargraph(n,m)\ur \planarclass(n,m)$ and assume $m=m(n)=\rounddown{\ratio n}$ for $\ratio \in(1,3)$. Then we have for all fixed $\niv,\nie\in\N_0$ 
\begin{align*}
\liminf_{n\to \infty}~ \prob{P \text{ has }\niv \text{ isolated vertices and }\nie \text{ isolated edges}}>0.
\end{align*}
\end{lem}
\begin{proof}
Throughout the proof, let $n$ be sufficiently large. Let $H$ be a fixed planar graph having $\niv$ isolated vertices and $\nie$ isolated edges and satisfying $\numberEdges{H}=\rounddown{\ratio\cdot \numberVertices{H}}$. Then we can construct \lq many\rq\ graphs in $\planarclass(n,m)$ with $\niv$ isolated vertices and $\nie$ isolated edges by adding a copy of $H$ to a connected graph $H'\in \planarclass_C(n-\numberVertices{H},m-\numberEdges{H})$. More precisely, we consider the following construction:
\begin{itemize}
\item Choose a subset $L\subseteq [n]$ of size $\numberVertices{H}$ and label the vertices of $H$ with $L$;
\item
Choose a graph $H'\in \planarclass_C(n-\numberVertices{H},m-\numberEdges{H})$, label the vertices with $[n]\setminus L$, and add the copy of $H$ to $H'$.
\end{itemize}
As these constructed graphs are all pairwise distinct, we obtain
\begin{align}\label{eq:29}
\prob{P \text{ has }\niv \text{ isolated vertices and }\nie \text{ isolated edges}}\geq \frac{\binom{n}{\numberVertices{H}}\left|\planarclass_C(n-\numberVertices{H},m-\numberEdges{H})\right|}{\left|\planarclass(n,m)\right|}.
\end{align}
We note that $\rounddown{\ratio \cdot\left(n-\numberVertices{H}\right)}$ is either $m-\numberEdges{H}$ or $m-\numberEdges{H}-1$. We assume $\rounddown{\ratio \cdot\left(n-\numberVertices{H}\right)}=m-\numberEdges{H}$, as the latter case can be done analogously by considering instead of $H$ a graph having $\numberVertices{H}$ vertices, $\numberEdges{H}+1$ edges, $k$ isolated vertices, and $l$ isolated edges. \Cref{thm:number_planar} implies
\begin{align}\label{eq:30}
\left|\planarclass_C(n-\numberVertices{H},m-\numberEdges{H})\right|=\Th{1}\left|\planarclass(n-\numberVertices{H},m-\numberEdges{H})\right|=\Th{1}n^{-\numberVertices{H}}\left|\planarclass(n,m)\right|.
\end{align}
Finally, the statement follows by using (\ref{eq:30}) and the fact 
$\binom{n}{\numberVertices{H}}=\Th{1}n^{\numberVertices{H}}$ in (\ref{eq:29}).
\end{proof}

We note that the formulas on the number of planar graphs in \Cref{thm:number_planar} are not true in the sparse regime where $m\leq n+\smallo{n}$. As a consequence, also \Cref{lem:isolated} does not hold in that case.

Next, we will show that we can locally change a graph so that the maximum degree increases by one, the number of isolated vertices by three, and the number of isolated edges decreases by two. Using it we can create graphs with many different maximum degrees. The following definition and lemma make this idea more precise.
\begin{definition}
For $n,m,\md \in \N$ and $\niv,\nie\in\N_0$ we let $\planarclass(n,m,\niv,\nie,\md)\subseteq\planarclass(n,m)$ be the subclass of all planar graphs on vertex set $[n]$ having $m$ edges, $\niv$ isolated vertices, $\nie$ isolated edges, and maximum degree $\md$.
\end{definition}

\begin{lem}\label{lem:increase_max_degree}
Let $n,m,\niv,\nie,\md\in\N$ with $\nie\geq 2$ and $\md\geq 3$. Then we have
\begin{align*}
\frac{\left|\planarclass(n,m,\niv+3,\nie-2,\md+1)\right|}{\left|\planarclass(n,m,\niv,\nie,\md)\right|}\geq \frac{1}{8k^3}.
\end{align*}
\end{lem}
\begin{proof}
We consider the following operation that transforms a graph $H\in \planarclass(n,m,\niv,\nie,\md)$ to a graph $H' \in \planarclass(n,m,\niv+3,\nie-2,\md+1)$ (see also \Cref{fig:double_counting}). We pick in $H$ a vertex $v_1$ of degree $\md$, a neighbour $v_2$ of $v_1$, an isolated vertex $v_3$, and isolated edges $v_4v_5$, $v_6v_7$. Then we obtain $H'$ from $H$ by deleting the edges $v_4v_5$, $v_6v_7$ and adding $v_1v_3$ and $v_2v_3$. For two graphs $H\in \planarclass(n,m,\niv,\nie,\md)$ and $H' \in\planarclass(n,m,\niv+3,\nie-2,\md+1)$ we write $\transformation{H}{H'}$ if $H$ can be transformed to $H'$ via the above operation. For a fixed graph $H$ we have
\begin{align}\label{eq:31}
\left|\setbuilder{H'}{\transformation{H}{H'}}\right|\geq \md\niv\binom{\nie}{2}.
\end{align}
Next, we note that if we perform our operation $\transformation{H}{H'}$, then $H'$ satisfies the following properties:
\begin{itemize}
\item There are at most two vertices in $H'$ with degree $\md+1$;
\item The vertex $v_1$ has degree $\md+1$;
\item The vertex $v_3$ has exactly two neighbours, which are $v_1$ and $v_2$;
\item The vertices $v_4, v_5, v_6$ and $v_7$ are isolated.
\end{itemize}
Using these observations we can bound for a fixed $H'$ the number of graphs $H$ with $\transformation{H}{H'}$. There are at most two possible vertices in $H'$ which can be $v_1$ and knowing $v_1$ there are at most $\md+1$ options for the vertex $v_3$. Given $v_1$ and $v_3$ the vertex $v_2$ is already determined. Finally, for the vertices $v_4, v_5, v_6$ and $v_7$ there are $3\binom{k+3}{4}$ possibilities. Hence, we obtain
\begin{align}\label{eq:32}
\left|\setbuilder{H}{\transformation{H}{H'}}\right|\leq 2(\md+1)3\binom{k+3}{4}.
\end{align}
Combining (\ref{eq:31}) and (\ref{eq:32}) yields
\begin{align*}
\frac{\left|\planarclass(n,m,\niv+3,\nie-2,\md+1)\right|}{\left|\planarclass(n,m,\niv,\nie,\md)\right|}\geq \frac{\md\niv\binom{\nie}{2}}{2(\md+1)3\binom{k+3}{4}}\geq \frac{1}{8k^3},
\end{align*}
where we used $k/\binom{k+3}{4}\geq 1/k^3$ and $d/(d+1)\geq 3/4$, as $d\geq 3$. This shows the statement.
\end{proof}
\begin{figure}[t]
	\begin{tikzpicture}[scale=1, line width=0.38mm, every node/.style={circle,fill=lightgray, inner sep=0, minimum size=0.39cm}]
		\node (1) at (0,0) [draw,fill,line width=0.5mm] {\footnotesize $v_1$};
		\node (2) at (1,1) [draw] {};
		\node (3) at (-1,1) [draw,fill,line width=0.5mm] {\footnotesize $v_2$};
		\node (4) at (1,-1) [draw] {};
		\node (5) at (-1,-1) [draw] {};
		\node (6) at (-2,0) [draw] {};
		\node (7) at (2,-0.7) [draw,line width=0.5mm] {\footnotesize $v_4$};
		\node (8) at (2,0.7) [draw,line width=0.5mm] {\footnotesize $v_5$};
		\node (9) at (2.8,-0.7) [draw,line width=0.5mm] {\footnotesize $v_6$};
		\node (10) at (2.8,0.7) [draw,line width=0.5mm] {\footnotesize $v_7$};
		\node (11) at (3.6,-0.7) [draw] {};
		\node (12) at (3.6,0.7) [draw] {};
		\node (13) at (4.4,-0.7) [draw] {};
		\node (14) at (4.4,0.7) [draw,line width=0.5mm] {\footnotesize $v_3$};
		\draw[-] (1) to (2);
		\draw[-] (1) to (3);
		\draw[-] (1) to (4);
		\draw[-] (1) to (5);
		\draw[-] (2) to (3);
		\draw[-] (2) to (4);
		\draw[-] (3) to (6);
		\draw[-] (5) to (6);
		\draw[-] (7) to (8);
		\draw[-] (9) to (10);
		\draw[-] (11) to (12);
		
		\node (1a) at (8.5,0) [draw,line width=0.5mm] {\footnotesize $v_1$};
		\node (2a) at (9.5,1) [draw] {};
		\node (3a) at (7.5,1) [draw,line width=0.5mm] {\footnotesize $v_2$};
		\node (4a) at (9.5,-1) [draw] {};
		\node (5a) at (7.5,-1) [draw] {};
		\node (6a) at (6.5,0) [draw] {};
		\node (7a) at (10.5,-0.7) [draw,line width=0.5mm] {\footnotesize $v_4$};
		\node (8a) at (10.5,0.7) [draw,line width=0.5mm] {\footnotesize $v_5$};
		\node (9a) at (11.3,-0.7) [draw,line width=0.5mm] {\footnotesize $v_6$};
		\node (10a) at (11.3,0.7) [draw,line width=0.5mm] {\footnotesize $v_7$};
		\node (11a) at (12.1,-0.7) [draw] {};
		\node (12a) at (12.1,0.7) [draw] {};
		\node (13a) at (12.9,-0.7) [draw] {};
		\node (14a) at (7.5,0) [draw,line width=0.5mm] {\footnotesize $v_3$};
		\draw[-] (1a) to (2a);
		\draw[-] (1a) to (3a);
		\draw[-] (1a) to (4a);
		\draw[-] (1a) to (5a);
		\draw[-] (2a) to (3a);
		\draw[-] (2a) to (4a);
		\draw[-] (3a) to (6a);
		\draw[-] (5a) to (6a);
		\draw[-] (11a) to (12a);
		\draw[-] (1a) to (14a);
		\draw[-] (3a) to (14a);
		
		\node (0) at (5.42,0) [draw=none, fill=none] {\Huge $\rightarrow$};
	\end{tikzpicture}
	\caption{The operation used in the proof of \Cref{lem:increase_max_degree}: A graph $H\in \planarclass(n,m,\niv,\nie,\md)$ is transformed to a graph $H' \in \planarclass(n,m,\niv+3,\nie-2,\md+1)$.}
	\label{fig:double_counting}
\end{figure}
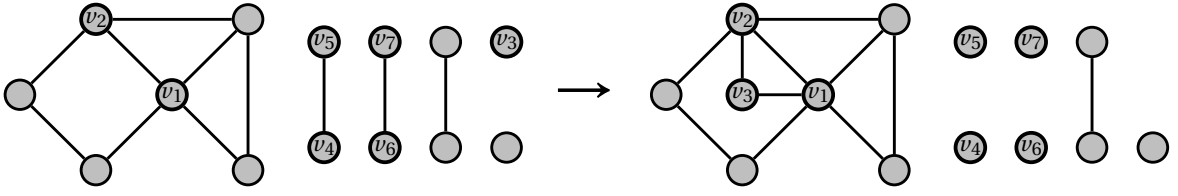
Finally, we can show \Cref{thm:main_dense} by using \Cref{lem:isolated} and applying \Cref{lem:increase_max_degree} repeatedly.
\proofof{thm:main_dense}
We assume to the contrary that there is a $A\in\N$ such that $\left|I\right|\leq A$ for infinitely many $n$. To simplify notation, we even assume that $\left|I\right|\leq A$ is true for all $n\in\N$. Otherwise, we could restrict our considerations to the subsequence consisting of all $n$ satisfying $\left|I\right|\leq A$. By \Cref{lem:isolated} there is a $\rho>0$ such that for all $n$ large enough
\begin{align*}
\sum_{d\in I}\left|\planarclass(n,m,1,2A,\md)\right|\geq \rho \left|\planarclass(n,m)\right|.
\end{align*}
In particular, we can choose $\md=\md(n)\in\N$ such that 
\begin{align*}
\left|\planarclass(n,m,1,2A,\md)\right|\geq \frac{\rho}{A} \left|\planarclass(n,m)\right|.
\end{align*}
Combining that with \Cref{lem:increase_max_degree} we get for all $i\leq A$
\begin{align*}
\left|\planarclass(n,m,1+3i,2A-2i,\md+i)\right|\geq \left(\frac{1}{8(3i-2)^3}\right)^i \left|\planarclass(n,m,1,2A,\md)\right|\geq \left(\frac{1}{8(3A)^3}\right)^A\frac{\rho}{A}\left|\planarclass(n,m)\right|.
\end{align*}
This implies that 
\begin{align*}
\prob{\maxdegree{P}=d+i}\geq \frac{\left|\planarclass(n,m,1+3i,2A-2i,\md+i)\right|}{\left|\planarclass(n,m)\right|}\geq \left(\frac{1}{8(3A)^3}\right)^A\frac{\rho}{A}.
\end{align*}
As $A$ and $\rho$ are fixed constants and \whp\ $\maxdegree{P}\in I$, this shows that for all $n$ large enough we have $d+i\in I$. Hence, we get $\left\{d, d+1, \ldots, d+A\right\}\subseteq I$, and therefore $|I|\geq A+1$, contradicting the fact $\left|I\right|\leq A$. This finishes the proof.
\qed

\section{Properties of $\concentration{\nbins}{\nballs}$}\label{sec:nu}
In this section we consider the function $\concentration{\nbins}{\nballs}$ introduced in \Cref{def:nu}. First we will show that $\concentration{\nbins}{\nballs}$ is well defined and then we will provide a proof of \Cref{lem:nu}.
\subsection{Well-definedness of $\concentration{\nbins}{\nballs}$}\label{sub:well_definedness}
We recall that for given $\nbins,\nballs\in \N$ we defined the function $f$ as
\begin{align}\label{eq:22}
	f(x)=f_{\nbins,\nballs}(x):=x\log\nballs+x-\left(x+1/2\right)\log x-(x-1)\log \nbins. 
\end{align}
By basic calculus we obtain $f(x)>0$ for all $x\in (0,1]$, $f''(x)<0$ for all $x\geq 1$, and $f(x)\to -\infty$ as $x\to \infty$. This implies that $f$ has a unique zero in $(0, \infty)$, which shows that $\concentration{\nbins}{\nballs}$ is well defined. 

Moreover, we obtain the following fact, which we will use in \Cref{sub:proof_nu}:
\begin{align}\label{eq:24}
	f(x)
	\begin{cases}
		>0 &\text{if } x<\concentration{\nbins}{\nballs},\\
		=0 &\text{if } x=\concentration{\nbins}{\nballs},\\
		<0 &\text{if } x>\concentration{\nbins}{\nballs}.
	\end{cases}
\end{align}
\subsection{Proof of \Cref{lem:nu}}\label{sub:proof_nu}
Let $f$ be as defined in (\ref{eq:22}) and let $\concentration{\nbins}{\nballs}$ be the unique positive zero of $f$. For $x\in(0,1]$ we have $f(x)>0$, which together with (\ref{eq:24}) implies \ref{lem:nu5}, i.e. $\concentration{\nbins}{\nballs}>1$. 

In order to prove \ref{lem:nu7}, we may assume that $\nballs\leq C\nbins^{1/3}$ for a suitable constant $C>0$. Now we get for $\nbins$ large enough
\begin{align*}
	f(5/3)\lessorequal -1/9\log n+5/3\log C+5/3-13/6\log\left(5/3\right)<0.
\end{align*}
Together with (\ref{eq:24}) this implies $\concentration{\nbins}{\nballs}\leq 5/3$ for all $\nbins$ large enough, which yields \ref{lem:nu7}.

For \ref{lem:nu1} we may assume $\nballs=\Th{\nbins}$. Then we have for $a>0$
\begin{align*}
	f\left(a\frac{\log n}{\log \log n}\right)=\left(1-a+\smallo{1}\right)\log n.
\end{align*}
Thus, \ref{lem:nu1} follows by (\ref{eq:24}). 

For \ref{lem:nu4} we fix $\nbins\in \N$ and define $K(x):=\left(1+1/(2x)\right)\log x+\left(1-1/x\right)\log n-1$. It is easy to check that $K(\concentration{\nbins}{\nballs})=\log \nballs$ and $K$ is strictly increasing. This implies \ref{lem:nu4}.

For \ref{lem:nu6} we observe that by definition of $\Concentration=\concentration{\nbins}{\nballs}$
\begin{align}\label{eq:2}
	1=e\frac{\nballs}{\nbins\Concentration}\exp\left(\left(\log \nbins-1/2\log \Concentration\right)/\Concentration\right).
\end{align}
Due to \ref{lem:nu1} and \ref{lem:nu4} we have $\Concentration=\smallo{\log n}$, which yields $\left(\log \nbins-1/2\log \Concentration\right)/\Concentration=\smallomega{1}$. Combining that with \eqref{eq:2} shows $\Concentration=\smallomega{\nballs/\nbins}$.

For \ref{lem:nu3} we let $\Concentration=\concentration{\nbins}{\nballs}$ and $\rho\in\R$. Due to \ref{lem:nu1} we have $\Concentration=\left(1+\smallo{1}\right)\log \nbins/\log\log \nbins$ and therefore
\begin{align}\label{eq:1}
	K\left(\Concentration+\rho\right)-K(\Concentration)=\frac{\left(\log\log n\right)^2}{\log n}\left(\rho +\smallo{1}\right).
\end{align}
On the other hand, we have 
\begin{align*}
	K\big(\concentration{\nbins}{\nballs+d}\big)-K\big(\concentration{\nbins}{\nballs}\big)=\log (\nballs+d)-\log\nballs=\Th{d/\nballs}=\smallo{\left(\log \log \nbins\right)^2/\log \nbins}.
\end{align*}
Together with (\ref{eq:1}) this implies \ref{lem:nu3}, as $K$ is strictly increasing.

Similarly, we define for \ref{lem:nu9} the function $g(x):=\left(x+1/2\right)\log x-x$. Now \ref{lem:nu9} follows by the facts that $g(\specialconcentration{\nbins})=\log\nbins$ and $g$ is strictly increasing. 

For \ref{lem:nu2}, let $C(x)=C_{n,k}(x):=\left(x+1/2\right)\log x+x\left(\log n-\log k-1\right)-\log n$ and $\rho \in \R$ be fixed. Then we have
\begin{align*}
	C(x+\rho)-C(x)=\rho\log(x+\rho)+\left(x+1/2\right)\log\left((x+\rho)/x\right)+\rho\left(\log n-\log k-1\right)=\rho\log(x+\rho)+\bigo{1}.
\end{align*}
Moreover, there is a constant $A>0$ independent of $n$ such that $C$ is strictly increasing for $x>A$. Due to \ref{lem:nu1} we have $\concentration{\nbins}{\nballs}=\smallomega{1}$ and $\concentration{c\nbins}{c\nballs}=\smallomega{1}$ and using the definition of $\Concentration$ we get $C\left(\concentration{c\nbins}{c\nballs}\right)-C\left(\concentration{\nbins}{\nballs}\right)=\log c-0=\bigo{1}$. This implies $\concentration{c\nbins}{c\nballs}=\concentration{\nbins}{\nballs}+\smallo{1}$.

Finally for \ref{lem:nu10}, let $x=\specialconcentration{n}$ and $y=\concentration{n}{cn}$. Using the definition of $\Concentration$ we have $0=x-(x+1/2)\log x+\log n=y\log c+y-(y+1/2)\log y+\log n$. Hence, we obtain
\begin{align}\label{eq:34}
	y-x=\frac{y\log c-\left(y+1/2\right)\log \left(y/x\right)}{\log x-1}.
\end{align}
By \ref{lem:nu1} we have $x,y=\left(1+\smallo{1}\right)\log n/\log\log n$ and $y/x=1+\smallo{1}$. Together with (\ref{eq:34}) this yields \ref{lem:nu10}.
\qed

\section{Discussion}\label{sec:discussion}
The only properties about the random planar graph $\planargraph=\planargraph(n,m)$ which we used in the proofs of our main results in the sparse regime (\Cref{thm:main_sub,thm:main} and \Cref{cor:maxdegree,cor:independent}) are the results on the internal structure in \Cref{thm:internal_structure}. Kang, Moßhammer, and Sprüssel \cite{surface} showed that \Cref{thm:internal_structure} is true for much more general classes of graphs. Prominent examples of such classes are cactus graphs, series-parallel graphs, and graphs embeddable on an orientable surface of genus $g\in \N\cup \{0\}$ (see \cite[Section 4]{cycles}). Using the generalised version of \Cref{thm:internal_structure} and analogous proofs of \Cref{thm:main_sub,thm:main} and \Cref{cor:maxdegree,cor:independent}, one can show the following.
\begin{thm}
\Cref{thm:main_sub,thm:main} and \Cref{cor:maxdegree,cor:independent} are true for the class of cactus graphs, the class of series-parallel graphs, and the class of graphs embeddable on an orientable surface of genus $g\in \N_0$.
\end{thm}

\Cref{thm:main} does not cover the whole regime $m=n+\smallo{n}$ and leaves a small gap of order $n^{\smallo{1}}$ to the dense case where $m=m(n)=\rounddown{\ratio n}$ for $\ratio \in(1,3)$. This leads to the following natural question on the behaviour of $\maxdegree{\planargraph}$ in the unconsidered region where $m=m(n)=n+t$ for a $t=t(n)>0$ such that $t=n^{1+\smallo{1}}$ and $t=\smallo{n}$.
\begin{question}
Let $\planargraph=\planargraph(n,m)\ur \planarclass(n,m)$ and $m=m(n)=n+t$ where $t=t(n)>0$ is such that $t=n^{1+\smallo{1}}$ and $t=\smallo{n}$. Is $\maxdegree{\planargraph}$ concentrated on a subset of $[n]$ with bounded size?
\end{question}

In \Cref{thm:main_dense} we saw that $\maxdegree{\planargraph}$ is not concentrated on any subset of $[n]$ with bounded size if $m=m(n)=\rounddown{\ratio n}$ for $\ratio \in(1,3)$. This raises the question how large a set $I=I(n)\subseteq[n]$ needs to be such that \whp\ $\maxdegree{\planargraph}\in I$ can hold. Furthermore, it would be interesting to know the precise asymptotic order of $\maxdegree{\planargraph}$ in that regime.
\begin{question}
Let $\planargraph=\planargraph(n,m)\ur \planarclass(n,m)$ and assume $m=m(n)=\rounddown{\ratio n}$ for $\ratio \in(1,3)$. What is the smallest size of a set $I=I(n)\subseteq[n]$ satisfying \whp\ $\maxdegree{\planargraph}\in I$? Moreover, what is the asymptotic order of $\maxdegree{\planargraph}$?
\end{question}

\section*{Acknowledgement}

The authors thank the anonymous referees for many helpful remarks to
improve the presentation of this paper.

\bibliographystyle{plain}
\bibliography{kang-missethan-max-degree}

\newpage
\appendix
\section{Sketch of the proof of \Cref{pro:ER}}\label{sec:proof_er}
Due to \Cref{thm:G_n_m_bins_balls}\ref{thm:G_n_m_bins_balls2} we have \whp\ 
\begin{align}\label{eq:35}
\maxdegree{G_i}=\concentration{n}{d_in}+\bigo{1}.
\end{align}
Together with \Cref{lem:nu}\ref{lem:nu10} this shows \whp\ $0<\maxdegree{G_2}-\maxdegree{G_1}=\Th{\log n/\left(\log \log n\right)^2}$. The so-called symmetry rule (see e.g. \cite[Section 5.6]{rg2}) says that $\Rest(G_i)$ \lq behaves\rq\ asymptotically like $G\left(n_i', m_i'\right)$ for suitable $n_i'=n_i'(n)=\Th{n}$ and $m_i'=m_i'(n)$ such that $2m_i'/n_i'$ tends to a constant $c_i'\in (0,1)$ with $c_1'>c_2'$. Together with \Cref{thm:G_n_m_bins_balls}\ref{thm:G_n_m_bins_balls2} this implies that \whp\
\begin{align}\label{eq:36}
\maxdegree{\rest{G_i}}=\concentration{n_i'}{2m_i'}+\bigo{1}=\concentration{n}{c_i'n}+\smallo{\log n/\left(\log \log n\right)^2},
\end{align}
where we used in the last equality \Cref{lem:nu}\ref{lem:nu2} and \ref{lem:nu10}. Using (\ref{eq:36}), \Cref{lem:nu}\ref{lem:nu10}, and the fact $c_1'>c_2'$, we obtain \whp\ $0<\maxdegree{\rest{G_1}}-\maxdegree{\rest{G_2}}=\Th{\log n/\left(\log \log n\right)^2}$. Combining \Cref{lem:nu}\ref{lem:nu10}, (\ref{eq:35}), and (\ref{eq:36}) yields that \whp\ $\maxdegree{\largestcomponent{G_i}}=\maxdegree{G_i}=\concentration{n}{d_in}+\bigo{1}$. Hence, we get \whp\ $0<\maxdegree{\largestcomponent{G_2}}-\maxdegree{\largestcomponent{G_1}}=\Th{\log n/\left(\log \log n\right)^2}$ and $0<\maxdegree{\largestcomponent{G_1}}-\maxdegree{\rest{G_1}}=\Th{\log n/\left(\log \log n\right)^2}$. \qed

\section{Proof of \Cref{thm:pruefer}}\label{sec:proof_pruefer}
We start by proving (\ref{eq:12}). To that end, let $r\in[\ntrees]$ be a root vertex. Throughout the construction of $\prueferseqence(\forest)$ the root $r$ is always the vertex with smallest label in the component of $\forest_{i}$ containing $r$. This implies $r\neq y_i$ for all $i\in[n-\ntrees]$. As the elements of the sequence $\mathbf{y}=\left(y_1, \ldots, y_{n-\ntrees}\right)$ are all distinct, we obtain 
\begin{align}\label{eq:21}
	\frequency{v}{\mathbf{y}}=
	\begin{cases}
		0 & \text{~~if~~} v\in [\ntrees]
		\\
		1 & \text{~~if~~} v\in[n]\setminus [\ntrees].
	\end{cases}
\end{align}
This proves (\ref{eq:12}), since $\degree{v}{\forest}=\frequency{v}{\mathbf{x}}+\frequency{v}{\mathbf{y}}$.

Next, we provide an algorithm that builds a graph $\prueferinvers(\mathbf{w})$ for each $\mathbf{w}\in \sequences{n}{\ntrees}$. Later we will see that the algorithm indeed reconstructs $\forest\in\forestclass(n,\ntrees)$ if the input is $\mathbf{w}=\prueferseqence(\forest)$. Let $\mathbf{w}=\left(w_1, \ldots, w_{n-\ntrees}\right)\in \sequences{n}{\ntrees}$ be given. We construct sequences $(\tilde{x}_1, \ldots, \tilde{x}_{n-\ntrees})$ and $(\tilde{y}_1, \ldots, \tilde{y}_{n-\ntrees})$ of vertices, a sequence $\left(\tilde{\forest}_0, \ldots, \tilde{\forest}_{n-\ntrees}\right)$ of forests and for each $v\in[n]$ a sequence $(\tilde{d}_0(v), \ldots, \tilde{d}_{n-\ntrees}(v))$ of degrees as follows. We start with $\vertexSet{\tilde{\forest}_0}=[n]$, $\edgeSet{\tilde{\forest}_0}=\emptyset$, $\tilde{d}_0(v)=\frequency{v}{\mathbf{w}}$ if $v\in [\ntrees]$, and $\tilde{d}_0(v)=\frequency{v}{\mathbf{w}}+1$ if $v\in [n]\setminus[\ntrees]$. For $i\in[n-\ntrees]$ we set $\tilde{x}_i=w_i$ and $\tilde{y}_i=\max\setbuilder{v}{\tilde{d}_{i-1}(v)=1}$. In addition, we let $\tilde{d}_i(v)=\tilde{d}_{i-1}(v)-1$ if $v\in\{\tilde{x}_i, \tilde{y}_i\}$ and $\tilde{d}_i(v)=\tilde{d}_{i-1}(v)$ otherwise. Finally, we obtain $\tilde{\forest}_i$ by adding the edge $\tilde{x}_i\tilde{y}_i$ in $\tilde{\forest}_{i-1}$ and we set $\prueferinvers(\mathbf{w}):=\tilde{\forest}_{n-\ntrees}$. Next, we show that this algorithm is well defined and that the output is indeed a graph. To that end, we note that for $v\in[n]\setminus[\ntrees]$ and $i\in[n-\ntrees]$ we have
\begin{align*}
	\big(\tilde{d}_{i-1}(v)\geq 1, \tilde{d}_i(v)=0\big)~ \implies~ \big(\tilde{y}_i=v\big).
\end{align*}
This yields that there are at least $(n-\ntrees-i)$ vertices $v\in[n]\setminus[\ntrees]$ with $\tilde{d}_i(v)\geq 1$. Thus, if $\tilde{d}_{i-1}(w_{n-\ntrees})\geq 1$ for some $i\in[n-\ntrees]$, then $\sum_{v\in [n]\setminus [\ntrees]}\tilde{d}_{i-1}(v)\leq 2(n-\ntrees-i)+1$ and therefore $\tilde{y}_i\in[n]\setminus[\ntrees]$. This yields $\tilde{d}_{i}(w_{n-\ntrees})\geq 1$ unless $i=n-\ntrees$. As $\tilde{d}_{0}(w_{n-\ntrees})\geq 1$ we obtain by induction that $\tilde{y}_i\in[n]\setminus[\ntrees]$ for all $i\in[n-\ntrees]$. In particular, this shows that $\tilde{y}_i$ is well defined and $\tilde{x}_i\neq \tilde{y}_i$. Thus, the algorithm is always executable and $\tilde{\forest}_{i}$ is a graph for all $i\in[n-\ntrees]$.

In order to prove that $\prueferseqence:\forestclass(n,\ntrees)\to \sequences{n}{\ntrees}$ is a bijection, it suffices to show the following claims.
\begin{enumerate}[label=(\roman*)]
	\item \label{cl:pruefer1}
	$\prueferseqence(\forest)\in \sequences{n}{\ntrees}$ for all $\forest\in \forestclass(n, \ntrees)$;
	\item \label{cl:pruefer2} $\prueferinvers(\prueferseqence(\forest))=\forest$ for all $\forest\in \forestclass(n, \ntrees)$;
	\item \label{cl:pruefer3}
	$\prueferinvers(\mathbf{w})\in\forestclass\left(n, \ntrees\right)$ for all $\mathbf{w}\in \sequences{n}{\ntrees}$;
	\item \label{cl:pruefer4}
	$\prueferseqence\left(\prueferinvers\left(\mathbf{w}\right)\right)=\mathbf{w}$ for all $\mathbf{w}\in \sequences{n}{\ntrees}$.
\end{enumerate}
We observe that $x_{n-\ntrees}\notin \left\{y_1, \ldots, y_{n-\ntrees}\right\}$. Thus, using (\ref{eq:21}) yields $x_{n-\ntrees}\in[\ntrees]$, which implies \ref{cl:pruefer1}.

To show \ref{cl:pruefer2}, we suppose that we first apply the algorithm to obtain $\prueferseqence(\forest)$ and then the algorithm $\prueferinvers$ with input $\mathbf{w}=\prueferseqence(\forest)$. Due to (\ref{eq:12}) the degree sequence of $\forest_0=\forest$ equals $\left(\tilde{d}_0(1), \ldots, \tilde{d}_0(n)\right)$ and therefore $\tilde{y}_1=y_1$. By construction we also have $\tilde{x}_1=x_1$, which implies that $\left(\tilde{d}_1(1), \ldots, \tilde{d}_1(n)\right)$ is the degree sequence of $\forest_1$. By repeating that argument we obtain by induction $\tilde{y}_i=y_i$ for all $i\in[n-\ntrees]$. As $\edgeSet{\forest}=\setbuilder{x_iy_i}{i\in[n-\ntrees]}$ and $\edgeSet{\tilde{\forest}_{n-\ntrees}}=\setbuilder{\tilde{x}_i\tilde{y}_i}{i\in[n-\ntrees]}$ this shows $\tilde{\forest}_{n-\ntrees}=\forest$, i.e. $\prueferinvers(\prueferseqence(\forest))=\forest$.

For \ref{cl:pruefer3} we assume that we apply the algorithm $\prueferinvers$ with input $\mathbf{w}\in \sequences{n}{\ntrees}$. By induction it follows that for all $i\in\left\{0, \ldots, n-\ntrees\right\}$ each component of $\tilde{\forest}_i$ contains at most one vertex $v$ with $\tilde{d}_i(v)>0$. This implies that we never close a cycle when adding the edge $\tilde{x}_{i+1}\tilde{y}_{i+1}$ in $\tilde{\forest}_i$, which shows that $\prueferinvers(\mathbf{w})$ is a forest. We saw before that $\tilde{y}_i\in[n]\setminus[\ntrees]$ for all $i\in[n-\ntrees]$. Thus, if $r\in[\ntrees]$ is a root and the component of $\tilde{\forest}_i$ containing $r$ has a vertex $v$ with $\tilde{d}_i(v)>0$, then $v=r$. This implies that adding the edge $\tilde{x}_{i+1}\tilde{y}_{i+1}$ never connects two components of $\tilde{\forest}_i$ which contain both a root. Hence, $\prueferinvers(\mathbf{w})\in \forestclass(n,\ntrees)$. 

Finally for \ref{cl:pruefer4}, we suppose that for given $\mathbf{w}\in \sequences{n}{\ntrees}$ we first apply the algorithm to construct $\prueferinvers(\mathbf{w})$ and then the algorithm to obtain the Prüfer sequence of $\prueferinvers(\mathbf{w})$. We note that the degree sequence of $\forest_0=\prueferinvers(\mathbf{w})$ equals $\left(\tilde{d}_0(1), \ldots, \tilde{d}_0(n)\right)$ and therefore $y_1=\tilde{y}_1$. By construction $\tilde{x}_1$ is the unique neighbour of $\tilde{y}_1$ in $\forest_0$, which implies $x_1=\tilde{x}_1$. This yields that the degree sequence of $\forest_1$ is $\left(\tilde{d}_1(1), \ldots, \tilde{d}_1(n)\right)$. Repeating that argument we obtain by induction $\tilde{x}_i=x_i$ for all $i\in[n-\ntrees]$, which proves \ref{cl:pruefer4}. \qed

\end{document}